\documentclass[12pt, reqno]{amsart}
\usepackage{amsmath, amsthm, amscd, amsfonts, amssymb, graphicx, color, mathrsfs}
\usepackage[bookmarksnumbered, colorlinks, plainpages]{hyperref}
\usepackage{cite}
\usepackage[all]{xy}
\usepackage{slashed}
\usepackage{tikz-cd}
\usepackage{mathabx}
\usepackage{tipa}
\usepackage{soul}
\usepackage{cancel}
\usepackage{ulem}
\usepackage{algorithm}
\usepackage{algpseudocode}
\usepackage{float}

\newtheorem{theorem}{Theorem}[section]
\newtheorem{lemma}[theorem]{Lemma}

\newtheorem{proposition}[theorem]{Proposition}
\newtheorem{corollary}[theorem]{Corollary}

\theoremstyle{definition}

\theoremstyle{remark}
\newtheorem{remark}[theorem]{Remark}
\numberwithin{equation}{section}

\begin{document}
\setcounter{page}{1}

\title[Strichartz estimates on compact Lie groups]{On the sharpness of Strichartz estimates and spectrum of compact Lie groups}

\author[D. Cardona]{Duv\'an Cardona}
\address{
  Duv\'an Cardona:
  \endgraf
  Department of Mathematics: Analysis, Logic and Discrete Mathematics
  \endgraf
  Ghent University, Belgium
  \endgraf
  {\it E-mail address} {\rm duvan.cardonasanchez@ugent.be}
  }

  \author[B. Grajales]{Brian Grajales}
\address{
  Brian Grajales:
  \endgraf
  Department of Mathematics
  \endgraf
  University of Campinas, Brazil
  \endgraf
  {\it E-mail address} {\rm grajales@ime.unicamp.br}
  }

\author[M. Ruzhansky]{Michael Ruzhansky}
\address{
  Michael Ruzhansky:
  \endgraf
  Department of Mathematics: Analysis, Logic and Discrete Mathematics
  \endgraf
  Ghent University, Belgium
  \endgraf
 and
  \endgraf
  School of Mathematical Sciences
  \endgraf
  Queen Mary University of London
  \endgraf
  United Kingdom
  \endgraf
  {\it E-mail address} {\rm michael.ruzhansky@ugent.be}
  }

\thanks{The authors were supported  by the FWO  Odysseus  1  grant  G.0H94.18N:  Analysis  and  Partial Differential Equations, by the Methusalem programme of the Ghent University Special Research Fund (BOF)
(Grant number 01M01021). B. Grajales has been partially supported by grant 2023/04083-0 (São Paulo Research Foundation FAPESP).  M. Ruzhansky is also supported  by EPSRC grant 
EP/R003025/2.
}


\begin{abstract} We prove Strichartz estimates on any compact connected simple  Lie group. In the diagonal case of Bourgain's exponents $p=q,$ we provide a new  regularity order $s_{0}^{R}(p)$ in the sense that our (reverse) Strichartz estimates are valid when   $s> s_{0}^{R}(p)$ and when $p\rightarrow 2^{+}.$ As expected our Sobolev index satisfies the estimate $ s_{0}^{R}(p)>s_{0}(d)=\frac{d}{2}-\frac{d+2}{p}.$ Motivated by the recent progress in the field, in the spirit of the analytical number theory methods developed by Bourgain in the analysis of periodic Schr\"odinger equations, we link the problem of finding Strichartz estimates on compact Lie groups with the problem of counting the number of representations $r_{s,2}(R)$ of an integer number $R$ as a sum of $s$ squares, and then, we implicitly use  the very well known bounds for $r_{s,2}(R)$ from the Hardy-Littlewood-Ramanujan circle method.  In our analysis, we explicitly compute the parametrisation of the spectrum of the Laplacian (modulo a factor depending on the geometry of the group) in terms of sums of squares. As a byproduct,  our approach also yields explicit expressions for  the spectrum of the Laplacian on all  compact connected simple  Lie groups, providing also a number of  results for Strichartz estimates in the borderline case $p=2.$
\end{abstract} 

\maketitle

\tableofcontents
\allowdisplaybreaks

\section{Introduction}

Through J.  Bourgain's work about dispersive equations on rational and irrational torus,  the role of the analytic number theory and in particular,  the analysis of the {\it Waring
 problem} via the recent complete solution in \cite{BourgainDemeter2015,BourgainDemeterGuth2016} of {\it Vinogradov's mean value conjecture}, has shown to be relevant for several problems in harmonic analysis and dispersive PDE.  In this work we follow the approach of Bourgain in order to investigate the extension of Strichartz estimates on an arbitrary simple compact Lie group $G$ (endowed with a bi-invariant Riemannian metric $g$) connecting this problem to the one of computing the number of representations $r_{s,2}(R)$ of an integer $R,$ as sums of $s\geq 2$ squares. 

\subsection{Outline}
Let $(M,g)$ be a closed Riemannian manifold. Let $u(t,x)=e^{it\Delta_g}$ be the solution of the Sch\"rodinger equation for the positive Laplacian $\Delta_g$ on a compact Riemannian manifold $(M,g).$   We record that an estimate of the type 
\begin{equation}\label{CS:E}
     \Vert u(t,x)\Vert_{L^p(I,L^q(M))}\leq C\Vert u_0\Vert_{H^{s}(M)},\,\,
 \end{equation} is called a {\it classical Strichartz estimate}, where $I\subset \mathbb{R}^+_0$ is a closed interval. It is well known that the Sobolev exponent $s$ in \eqref{CS:E} is expected to satisfy the inequality
 \begin{equation}
     s\geq s_{0}(p,q)=\frac{d}{2}-\frac{2}{p}-\frac{d}{q}.
 \end{equation}

 The validity of this kind of Strichartz estimates is very well known from the work of Burq, G\'erard and Tzvetkov \cite{BurqGerardTzvetkov2004}. Their estimate is as follows.
 
 {\it Let $d\geq 1.$  The solution to the linear Schrodinger equation
 \begin{equation}\label{SchEq}
     iu_t+\Delta_gu=0,\,\,u(0,x)=u_0(x),\,\,(t,x)\in [0,1]\times M,
 \end{equation} satisfies the estimate
 \begin{equation}\label{BGT}
     \Vert u(t,x)\Vert_{L^p([0,1],L^q(M))}\leq C\Vert u_0\Vert_{H^{\frac{1}{p}}(M)},\,\,
 \end{equation}where $(p,q)$ are Lebesgue indices satisfying the admissibility condition  $$2/p+d/q=d/2, \,\,(p,q,d)\neq (2,\infty,2).$$} 
 The optimality of \eqref{BGT} in the end-point case $p=2$ was proved in \cite{BurqGerardTzvetkov2004} for spheres $\mathbb{S}^{d}.$ Moreover, they have proved the following general statement when adapting Bourgain's argument to this generality level.

 { \it Let $d\geq 2.$  Assume that $(M,g)$ is a compact Riemannian manifold whose geodesics are closed and with a common period. Then the estimate \eqref{BGT} can be improved as
 \begin{equation}\label{BGT:2}
     \Vert u(t,x)\Vert_{L^4([0,1]\times M)}\leq C\Vert u_0\Vert_{H^{s}(M)},\,\,s>s_0(d),\,\,
 \end{equation}where 
 \begin{equation}\label{BGT:s(d)}
     s_0(2)=1/8,\,s_0(d)=\frac{d}{4}-\frac{1}{2},\, \forall d\geq 3.
 \end{equation}Moreover, if $M=\mathbb{S}^d$ is endowed with the canonical metric, the Sobolev regularity order $s>s_0(d)$ in \eqref{BGT:2} is sharp. 
 } 
 
 We note that the the regularity order in \eqref{BGT:s(d)} agrees with the one obtained by Bourgain in the case of the torus only when $d\geq 3.$ Indeed, on the flat torus $\mathbb{T}^d,$ and by computing explicitly $$\Vert u\Vert_{L^4(\mathbb{T}\times \mathbb{T}^d)}^2=\Vert u^2\Vert_{L^2(\mathbb{T}\times \mathbb{T}^d)}.$$ Bourgain showed that:

 {\it if $d=2,$ for every $\varepsilon>0,$ $\Vert u\Vert_{L^4(\mathbb{T}\times \mathbb{T}^d)}\leq C_\varepsilon\Vert u_0\Vert_{H^\varepsilon(\mathbb{T}^d)}, $ and when $d\geq 3,$ the following estimate
 \begin{equation}\label{Bourgain}
     \Vert u(t,x)\Vert_{L^4([0,1]\times M)}\leq C\Vert u_0\Vert_{H^{s}(M)},\,\,s>s_0(d),\,\,
 \end{equation}holds for $s_0(d)={d}/{4}-{1}/{2}.$ 
 }

Due to the general perspective of the results in \eqref{BGT} or in \eqref{BGT:2} the sharpness of these results is an interesting problem. The best-studied situation is when $M=\mathbb{R}^d/\Gamma$ is a torus for some rank $d$ lattice. Improvements of \eqref{Bourgain} for general indices of $p$ is a significant problem due to its connections with {\it weak turbulence} for non-linear Schr\"odinger equations, see  e.g.  \cite[Page 703]{DengGermainGuthRydinMyerson}. 

Recently, in a series of works, Zhang in \cite{ZhangThesis,Zhang2020,Zhang2020:2,Zhang2021} has obtained a  number of improvements in the setting of simple compact Lie groups  for Strichartz estimates of the Schr\"odinger equation by obtaining analogues of \eqref{BGT:2} and extensions of it for all $1<p<\infty.$ Zhang's result can be stated as follows, see \cite[Theorem 1]{Zhang2021} in the compact Lie group setting.

{\it Let $(G,g)$ be a compact Lie group  of rank $\ell\geq 2$ equipped with the canonical
Killing metric $g$ and let $I\subset \mathbb{R}^{+}_0$ be a closed interval. Then the solution to \eqref{SchEq} satisfies the estimate
\begin{equation}\label{Zhang}
    \Vert u(t,x)\Vert_{L^p(I\times G)}\lesssim  \Vert u \Vert_{H^s(G)},
\end{equation}provided that $p\geq 2+\frac{8}{\ell},$ for all $s$ such that
\begin{equation}\label{Zhang:2}
  s\geq \frac{d}{2}-\frac{d+2}{p}.  
\end{equation}
 }

{{ 
The aim of this work is to analyse the case of reverse Strichartz estimates on a compact Lie group $G$. By {\it reverse Strichartz estimates} on a compact Riemannian manifold $(M,g)$ we mean Sobolev estimates of the form 
\begin{equation}\label{CS:E}
     \Vert u(t,x)\Vert_{L^q(M,\,L^p(I))}\leq C\Vert u_0\Vert_{H^{s}(M)}.\,\,
 \end{equation} The main goal of this work is to estimate the Sobolev index $s_{0}^{R}(p,q)$ that allows the validity of \eqref{CS:E} when $s\geq s_{0}^{R}(p,q). $ The Sobolev index $s_{0}^{R}(p,q)$ is expected to hold  the inequality 
 \begin{equation}
     s_{0}^{R}(p,q)\geq  s_{0}(p,q)=\frac{d}{2}-\frac{2}{p}-\frac{d}{q}.
 \end{equation} In particular, when $p=q$ we expect to provide a new value for the index $s_{0}^{R}(p,p),$ just when $p\in [2,2+\varepsilon]$ with $\varepsilon>0$ small enough. Indeed, with $\varepsilon=2+\frac{4(d+\ell)}{d\ell},$ is known from the works of Zhang \cite{ZhangThesis,Zhang2020,Zhang2020:2,Zhang2021} that one can take $s_{0}^{R}(p,p)=  s_{0}(p,p),$ if $p\in [2+\frac{4(d+\ell)}{d\ell},\infty).$  }} 
 
 Here, we consider only the case where $G$ is a simple and compact Lie group, because we make use of explicit representations of the spectrum of the Laplace operator, which modulo a factor, can be written in terms of a sum of squares. Indeed, it is a very well-known fact that for a fixed set of simple positive roots $\Sigma$ with corresponding fundamental weights $\overline{\omega}_1,...,\overline{\omega}_n,$ and each highest weight $$\small{{\tilde{\nu}}=\sum\limits_{j=1}^l{\tilde{\nu}}_j\overline{\omega}_j},\,{\tilde{\nu}}_j\in\mathbb{N}_0,$$ on a compact Lie group $G,$ one can associate an eigenvalue $\lambda_{{\tilde{\nu}}}$ for the positive Laplacian $\mathcal{L}_G$ given by 
 \begin{equation}\label{eigenvalue:formula:2:intro}
    \lambda_{{\tilde{\nu}}}=(\eta,\eta)-({\tilde{\nu}}+\eta,{\tilde{\nu}}+\eta)\footnote{where
 $ 
    \eta:=\frac{1}{2}\sum\limits_{\alpha\in\Pi^+}\alpha,
$ and the sum runs over the system of positive roots $\Pi^+.$ },
\end{equation}
Conversely, every eigenvalue of $\mathcal{L}_G$ can be written in the form of  \eqref{eigenvalue:formula:2:intro}. Due to the importance that it has for us the explicit construction of the spectrum of the Laplace operator in terms of the root systems, we compute it on every  simple and simply connected $G.$ Since these groups  are classified by the root system associated with their Cartan subalgebras, in terms of their Dynkin diagrams, we split our analysis into the following cases:
 \begin{itemize}
     \item the Lie algebra $\mathfrak{g}$ of $G$ is $A_\ell, $ in which case $G\cong \textnormal{SU}(\ell+1).$
     \item The Lie algebra $\mathfrak{g}$ of $G$ is $B_\ell, $ in which case $G\cong \textnormal{Spin}(2\ell+1).$
     \item The Lie algebra $\mathfrak{g}$ of $G$ is $C_\ell, $ in which case $G\cong \textnormal{Sp}(\ell).$
     \item  The Lie algebra $\mathfrak{g}$ of $G$ is $D_\ell, $ in which case $G\cong \textnormal{Spin}(2\ell).$
     \item  The Lie algebra $\mathfrak{g}$ of $G$ is one of the exceptional Lie algebras $\mathfrak{g}_2,\mathfrak{f}_4 $ or $\mathfrak{e}_8,$ in which case  $G\cong G_2,$ $G\cong F_4,$ or $G\cong E_8,$ respectively.    
 \end{itemize}
 For our further analysis, we exclude the case  $\ell=1$  since $\textnormal{Spin(3)}\cong \textnormal{SU}(2)\cong \mathbb{S}^3$ and the Strichartz estimates in the case of the spheres have been analysed in the work of Burq, Gerard, and Tzvetkov \cite{BurqGerardTzvetkov2004}. Having explained this classification, we are ready to present our result. 
 \subsection{Main result}
 In the case where $2\leq p<\infty$ and $2\leq q\leq \infty,$ the results of this paper can be summarised in Theorem \ref{Main:Theorem} below. In view of Remark \ref{nonsimply:connected},   we only  consider the validity of the Strichartz estimates in the case where $G$ is a connected simple and simply connected compact Lie group. However, our analysis also covers  the case where $G$ is not simply connected. Indeed,  if $G$ is a connected (not necessarily simply-connected) compact simple Lie group, we have that $$G\cong \tilde{G}/\Gamma,$$ where $\tilde{G}$ is a connected simply-connected compact simple Lie group (diffeomorphic to the universal cover of $G$) and $\Gamma$ is a discrete subgroup of $\tilde{G}.$ This property allows us to conclude the inclusion \begin{equation}
  \textnormal{Spec}(G)\subseteq\textnormal{Spec}(\tilde{G}).
\end{equation}As in our approach we apply an explicit construction of the spectrum of the Laplacian we have proved the following theorem.
 
\begin{theorem}\label{Main:Theorem} Let $2\leq p<\infty.$ Let $(G,g)$ be a compact connected simple Lie group of rank $\ell\geq 2$ endowed with a Riemmanian bi-invariant metric $g$ and let $\tilde{G}$ be its universal cover.  Let 
\begin{equation}\label{m}
   m=\left\{\begin{array}{lll}
\ell+1, &&\textnormal{if}\ \tilde{G}\cong\textnormal{SU($\ell+1$)}\ \textnormal{or}\ \tilde{G}\cong G_2,\\
\ell, &&\textnormal{otherwise}.
\end{array}\right. 
\end{equation}
Let $T=2\pi b  \gamma,$ where $b$ is a non-negative integer depending on $G$ and $\gamma$ is  a geometric constant defined by the metric $g.$ Then 
\begin{equation}\label{Intro:maisn:SE}
\left|\left|u\right|\right|_{L^q(G,L^p[0,2\pi b  \gamma])}\lesssim_s  \Vert u_0\Vert_{H^{ s} (G) },
\end{equation} 
where $$s \geq s_{0}^R(p,q):=   \frac{m}{2}-\frac{2}{p}+\sigma(q)$$ if $m\geq 5,$ $s> s_{0}^R(p,q)$ if $m=2,3,4$ and 
\begin{equation*}
    \sigma(q):=\left\{\begin{array}{lll}
     \frac{d-1}{2}\left(\frac{1}{2}-\frac{1}{q}\right),&& \textnormal{if}\ 2\leq q\leq \frac{2(d+1)}{d-1},\\
     \\
     \frac{d-1}{2}-\frac{d}{q},&& \textnormal{if}\ \frac{2(d+1)}{d-1}<q\leq {\infty.}
     \end{array}\right.
\end{equation*} 
\end{theorem}
Now, we briefly discuss this result and its relation to the existent results on compact manifolds. 
\begin{remark}
    Note that when $\Tilde{G}\cong \textnormal{SU}(\ell+1)$ or if $\Tilde{G}\cong \textnormal{G}_2,$ the root system of the Lie algebra is constructed in an Euclidean space of dimension $\ell+1.$ Also, if $\Tilde{G}$ is not diffeomorphic to $ \textnormal{SU}(\ell+1)$ or to $\textnormal{G}_2,$ its root system is constructed in an Euclidean $\ell$-dimensional space. These facts justify the two values $m=\ell,$ or $m=\ell+1,$ of $m$ defined in \eqref{m}.
\end{remark}
\begin{remark} On $\mathbb{R}^n,$ the sharp Strichartz estimate
\begin{equation}\label{Strichartz}
    \Vert u(t,x)\Vert_{L^p_tL^q_x(\mathbb{R}\times \mathbb{R}^d )}\leq C\Vert u_0\Vert_{L^2(\mathbb{R}^d)},
\end{equation}is valid for all the {\it admissible pairs} $(p,q),$ i.e. $p$ and $q$ satisfying $2/p+d/q=d/2,$ see Ginibre and Velo \cite{GinibreVelo} and Keel and Tao \cite{KeelTao98}. Note that by the Sobolev embedding theorem one has that the estimate
\begin{equation}\label{1}
    \Vert u(t,x)\Vert_{L^p_tL^r_x(\mathbb{R}\times \mathbb{R}^d )}\leq C\Vert u_0\Vert_{H^s(\mathbb{R}^d)},
\end{equation}is valid with
\begin{equation}\label{2}
    s=\frac{d}{2}-\frac{2}{p}-\frac{d}{r}\geq 0,
\end{equation} where $p,q\geq 2,$ and $(p,r,d)\neq (2,\infty,2).$ The triples $(p,q,r)$ satisfying  \eqref{1} and \eqref{2} are usually called {\it scale-invariant}. However, on compact manifolds the Strichartz estimates as in \eqref{Strichartz} are expected to fail, while in the setting of compact manifolds the validity of an estimate of the type \eqref{1}
is expected to hold only for positive values of $s>0.$ Moreover, is also expected to have Strichartz estimates for $(p,q,r)$ being a non-admissible triple and with $s>\frac{d}{2}-\frac{2}{p}-\frac{d}{r},$ see e.g. the discussion in Zhang \cite{Zhang2020}.   
\end{remark}
\begin{remark}\label{Remark:sh}
Recently Zhang in \cite{Zhang2023} has improved the validity of Strichartz estimates    \eqref{Zhang} where $s\geq \frac{d}{2}-\frac{d+2}{p}$ and    $p\geq 2+\frac{8(s_0-1)}{s_0\ell},$ where $s_0=2d/(d-\ell).$ Note that for $s=\frac{d}{2}-\frac{d+2}{p}$ one has that $p\geq  2+\frac{4(d+\ell)}{d\ell}.$  So, the contributions of our result for $p$ in the interval $$[2,2+\frac{4(d+\ell)}{d\ell}]$$ are  new.  
{{If we denote by
\begin{equation}
 s_{0}^R(p,q):=   \frac{m}{2}-\frac{2}{p}+\sigma(q)
\end{equation}  the Sobolev exponent obtained in Theorem \ref{LemmaSpin2l:1},
 one has the inequality $s_{0}^R(p,q)\geq s_{0}(p,q).$ Indeed, for $q\in [2,\frac{2(d+1)}{d-1}],$ one can prove that
$$ s_{0}^R(p,q)-s_{0}(p,q)=\frac{m}{2}-\frac{d+1}{2}\left(\frac{1}{2}-\frac{1}{q}\right) \geq 0, $$ while for $2\in \left(\frac{2(d+1)}{d-1},\infty\right)$ one has that
$$ s_{0}^R(p,q)-s_{0}(p,q)=\frac{m-1}{2}\geq 0. $$
 We do not expect that the index  $s_{0}^R(p,q)$ will be sharp when $p\rightarrow 2^+,$ but still when $q=2,$ the index $s_{0}^R(p,2)$ is sharp with respect to the approach used in this work.}}

\end{remark}

\section{Tools from Lie theory and from analytic number theory}

\subsection{Spectrum of the Laplacian on compact simple Lie groups}
We will find the spectrum of the Laplacian operator on a Lie group via representations of its Lie algebra and provide some explicit examples of the eigenvalues in compact and simply-connected Lie groups. For this, we shall follow \cite{BeSvir2010,BeZuSvir2016,Svirkin2010,Zubareva2017,Zubareva2020}. 

Let $G$ be a $n$-dimensional connected and simply connected compact simple Lie group whose corresponding Lie algebra $\mathfrak{g}$ is the real compact form of a complex simple Lie algebra $\mathfrak{k}.$ Let us endow $\mathfrak{g}$ with the $\textnormal{Ad}$-invariant inner product $g$ given by a multiple $\gamma>0$ of the negative of the Killing form of $\mathfrak{g}$, that is, $g=-\gamma\kappa,\ \textnormal{where}\ \kappa(X,Y):=\textnormal{Tr(ad(}X\textnormal{)ad(}Y\textnormal{))}$. It is known that $g$ induces a bi-invariant Riemannian metric on $G$ which will be also denoted by $g$. The Laplacian on $G$ is defined by $$\mathcal{L}_G:=-\sum\limits_{j=1}^nX_j^2,$$ where $\mathbb{X}=\{X_1,...,X_n\}$ is a $g$-orthonormal basis of left-invariant vector fields. Let  $\mathfrak{t}$ be a Cartan subalgebra of $\mathfrak{k}$ (i.e. $\mathfrak{t}$ is nilpotent and self-normalising) and denote by $\mathfrak{t}(\mathbb{R})$ its corresponding real form so that we can consider the root system  $$\Pi:=\{\alpha\in\mathfrak{t}(\mathbb{R})^*\setminus\{0\}|\exists X\in\mathfrak{g}\setminus\{0\}\ \textnormal{such that}\ [H,X]=\alpha(H)X,\ \forall H\in\mathfrak{t}(\mathbb{R})\}\subseteq\mathfrak{t}(\mathbb{R})^*.$$
Fix subsets $\Pi^+\subseteq\Pi$ and $\Sigma=\{\alpha_1,...,\alpha_\ell\}\subseteq\Pi^+$ of positive roots and positive simple roots respectively. The restriction $(\cdot,\cdot):=\gamma\kappa\left|_{\mathfrak{t}(\mathbb{R})\times\mathfrak{t}(\mathbb{R})}\right.$ on $\mathfrak{t}(\mathbb{R})$ is non-degenerate and it induces a quadratic form on $\mathfrak{t}(\mathbb{R})^*$ (also denoted by $(\cdot,\cdot)$) as follows: since $(\cdot,\cdot)$ is non-degenerate, then the map $H\mapsto\alpha_H:=(H,\cdot)$ is an isomorphism between $\mathfrak{t}(\mathbb{R})$ and $\mathfrak{t}(\mathbb{R})^*$; this means that for each $\alpha\in\mathfrak{t}(\mathbb{R})$ there exists a unique $H_\alpha\in\mathfrak{t}(\mathbb{R})$ satisfying $\alpha(H)=(H_\alpha,H),\ \forall H\in\mathfrak{t}(\mathbb{R})$, so we define $(\alpha,\beta):=(H_\alpha,H_\beta).$ 

The fundamental weights $\overline{\omega}_1,...,\overline{\omega}_\ell\in\mathfrak{t}(\mathbb{R})^*$ of $\mathfrak{g}$ with respect to $\Sigma$ are defined by the relations
$$(\overline{\omega}_i,\alpha_j)=\delta_{ij}\frac{(\alpha_j,\alpha_j)}{2},\ i,j=1,...,\ell,$$
where $\delta_{ij}$ is the Kronecker symbol. 
For a given a complex finite dimensional representation $\rho:\mathfrak{k}\rightarrow\mathfrak{gl}(V)$, there exists a decomposition
$$V=\bigoplus\limits_{\alpha\in\Phi_\rho}V_\alpha,$$
where  $V_\alpha:=\{v\in V:\rho(H)v=\alpha^{\mathbb{C}}(H)v, \forall H\in\mathfrak{t}\}$ and $\Phi_\rho:=\{\alpha\in\mathfrak{t}(\mathbb{R})^*:V_\alpha\neq\{0\}\}.$ The elements of $\Phi_\rho$ and the subspaces $V_\alpha,\ \alpha\in\Phi_\rho,$ are called the {\it weights} and the {\it weight subspaces} of the representation $\rho$ respectively. If $\alpha,\beta\in\Phi_\rho$, we say that $\alpha$ is {\it higher} than $\beta$ whenever $\alpha-\beta$ is a linear combination of positive roots with non-negative coefficients. A  {\it highest weight} of $\rho$ is a weight $\tilde{\nu}$ which is higher than every other weight. The following is classical result by Cartan \cite{Cartan}.
\begin{theorem}[Theorem of the highest weight]\label{Cartan:theorem} With the notations above we have that the following assertions hold:

\begin{itemize}
\item[(1)] Every finite-dimensional irreducible complex representation $\rho:\mathfrak{k}\rightarrow\mathfrak{gl}(V)$ has a unique highest weight $\tilde{\nu}$ which can be written as a linear combination of the fundamental weights $\overline{\omega}_1,...,\overline{\omega}_\ell$ with integral non-negative coefficients, that is,
\begin{equation}\label{highest:weight}\tilde{\nu}=\sum\limits_{j=1}^\ell\tilde{\nu}_j\overline{\omega}_j,\ \tilde{\nu}_j\in\mathbb{N}_0.
\end{equation}
Moreover, two irreducible finite-dimensional representations of $\mathfrak{k}$ with the same highest weight are equivalent.
\item[(2)] Every functional of the form \eqref{highest:weight} is the highest weight of an irreducible complex representation of $\mathfrak{k}.$
\end{itemize}
\end{theorem}
It is a known fact that if $\theta:G\rightarrow GL(V)$ is an irreducible complex finite-dimensional representation of the group $G$, then the function \begin{equation}\label{character}\varphi=\frac{1}{d}\chi_\theta,\end{equation} where $d=\textnormal{dim}(V)$ and $\chi_\theta$ is the character of the representation $\theta,$ is an eigenfunction of $\mathcal{L}_G$ and conversely, for each $\lambda\in\textnormal{Spec}(\mathcal{L}_G)\setminus\{0\}$ there exists a eigenfunction of the form \eqref{character} associated to the eigenvalue $\lambda.$ Furthermore, for each $\lambda\in\textnormal{Spec}(\mathcal{L}_G)\setminus\{0\}$, there exists a real irreducible representation $\theta_0:G\rightarrow GL(V)$ such that $\varphi=\frac{1}{d}\chi_{\theta_0}$ is an eigenfunction of $\mathcal{L}_G$ corresponding to $\lambda$ (see \cite[Corollary 3.2]{BeSvir2010}). The next proposition provides a formula to compute $\textnormal{Spec}(\mathcal{L}_G)$ in terms of real irreducible representations of $G.$ We refer to \cite[Theorem 3.2]{BeSvir2010} for a proof.

\begin{proposition}\label{eigenvalue:theorem}
Let $\theta_0:G\rightarrow GL(V)$ be an irreducible real $d$-dimensional representation of $G$ and $\rho_0:=(d\theta_0)_e:\mathfrak{g}\rightarrow \mathfrak{gl}(V)$ its corresponding infinitesimal representation. Then
\begin{equation}\label{eigenvalue:formula}
\lambda_{\theta_0}:=-\frac{1}{d}\sum\limits_{j=1}^n\textnormal{Tr}[\rho_0(X_j)^2]\in\textnormal{Spec}(\mathcal{L}_G).
\end{equation}  
\end{proposition}
 Since we are assuming that $G$ is simply connected, then the restriction $\rho\left|_{\mathfrak{g}}\right.$ of an irreducible complex representation $\rho$ of $\mathfrak{k}$  induces an irreducible real representation $\rho_0$ of $\mathfrak{g}$ whose character is uniquely determined by the highest weight $\tilde{\nu}$ of $\rho.$ On the other hand, irreducible real finite-dimensional representations of the Lie algebra $\mathfrak{g}$ are in bijective correspondence with irreducible real finite-dimensional representations of the group $G.$ So we can parametrise the nonzero eigenvalues of $\mathcal{L}_G$ with the set of all highest weights $\tilde{\nu}$ of irreducible complex finite-dimensional representations of $\mathfrak{k}$ which, due to Theorem \ref{Cartan:theorem}, is nothing but the set of all linear combinations of the fundamental weights with non-negative integral coefficients. Moreover, we have the following proposition:
 \begin{proposition}\textnormal{(\cite[Theorem 5.20]{BeSvir2010})} For each highest weight $\tilde{\nu}=\sum\limits_{j=1}^l\tilde{\nu}_j\overline{\omega}_j,\ \tilde{\nu}_j\in\mathbb{N}_0,$ the associated eigenvalue of $\mathcal{L}_G$ is given by 
 \begin{equation}\label{eigenvalue:formula:2}
    \lambda_{\tilde{\nu}}=(\eta,\eta)-(\tilde{\nu}+\eta,\tilde{\nu}+\eta),
\end{equation}
where
\begin{equation*}
    \eta:=\frac{1}{2}\sum\limits_{\alpha\in\Pi^+}\alpha.
\end{equation*}
\end{proposition}
\begin{remark}\label{nonsimply:connected}
It is well-known that if $p:M\rightarrow N$ is a Riemannian submersion with totally geodesic fibers and $\varphi$ is an eigenfunction of the Laplacian operator in $N$, then $\varphi\circ \pi$ is an eigenfunction of the Laplacian operator in $M$ corresponding to the same eigenvalue \cite{Berger1971}. As a consequence of this fact we have that $\textnormal{Spec}(N)\subseteq\textnormal{Spec}(M)$. In particular, if $G$ is a connected (not necessarily simply-connected) compact simple Lie group, then $G\cong \tilde{G}/\Gamma$, where $\tilde{G}$ is a connected simply-connected compact simple Lie group (diffeomorphic to the universal cover of $G$) and $\Gamma$ is a discrete subgroup of $\tilde{G}$,  every bi-invariant metric $g$ on $G$ is induced by a bi-invariant metric $\tilde{g}$ on $\tilde{G}$ in such a way that the projection $\pi:\tilde{G}\rightarrow \tilde{G}/\Gamma\cong G$ is a Riemannian submersion with totally geodesic fibers. This allows us to conclude that, in this situation we have \begin{equation}\label{spectrum:comparision}
  \textnormal{Spec}(G)\subseteq\textnormal{Spec}(\tilde{G}),
\end{equation}
\end{remark}
With the notations in \cite[Chapter 12]{Humphreys}, we can construct all the irreducible root systems starting from an Euclidean space endowed with an inner product $\langle\cdot,\cdot\rangle$ and a $\langle\cdot,\cdot\rangle$-orthonormal basis $\{\varepsilon_j\}$. From this construction we have that $(\cdot,\cdot)=-\frac{1}{b\gamma}\langle\cdot,\cdot\rangle$ with $b=\langle\tilde{\alpha},\tilde{\alpha}\rangle+2\langle\tilde{\alpha},\eta\rangle$ and $\tilde{\alpha}$  the highest long root so that 
\begin{equation}\label{eigenvalue:formula:3}
    \lambda_{\tilde{\nu}}=(\eta,\eta)-(\tilde{\nu}+\eta,\tilde{\nu}+\eta)= \frac{1}{b\gamma}\left[\langle\tilde{\nu}+\eta,\tilde{\nu}+\eta\rangle-\langle\eta,\eta\rangle\right].
\end{equation}
We shall use the formula in \eqref{eigenvalue:formula:3} to compute explicitly the eigenvalues of $\mathcal{L}_G$ for the connected and simply-connected compact simple Lie groups classified by their root system.\\

\noindent $\bullet$ {\it Diagram} $A_\ell.$ In this case, one has that $G\cong \textnormal{SU}(\ell+1).$ The spectrum of the Laplacian is computed in the following theorem. 
\begin{theorem}\label{spectrum:theorem:Al} The spectrum of the Laplacian operator on \textnormal{SU($\ell+1$) ($\ell\geq 1$)} is given by
\begin{equation*}
\textnormal{Spec(SU(}\ell+1\textnormal{))}=\{\lambda_\nu:\nu=(\nu_1,...,\nu_\ell)\in\mathbb{N}^\ell\},
\end{equation*}
where
\begin{align}
    \lambda_\nu=&\frac{1}{2(\ell+1)^3\gamma}\left[\left(\sum\limits_{j=1}^\ell(\ell+1-j)\nu_j\right)^2\right.+\sum\limits_{k=2}^\ell\left(-\sum\limits_{j=1}^{k-1}j\nu_j+\sum\limits_{j=k}^\ell(\ell+1-j)\nu_j\right)^2 \nonumber\\
    &\hspace{2cm}+\left.\left(\sum\limits_{j=1}^\ell j\nu_j\right)^2-\frac{1}{4}\sum\limits_{k=1}^{\ell+1}(\ell-2(k-1))^2(\ell+1)^2\right].
\end{align}
\end{theorem}
\begin{proof} The root system associated to SU($\ell+1$) is $A_\ell$. The standard realisation of $A_\ell$ is given by a $(\ell+1)-$dimensional Euclidean space with an inner product $\langle\cdot,\cdot\rangle$ and an orthonormal basis $\{\varepsilon_1,...,\varepsilon_{\ell+1}\}$. The sets of positive roots and simple positive roots are $$\Pi^+=\{\varepsilon_i-\varepsilon_j:1\leq i<j\leq \ell+1\}\ \textnormal{and}\ \Sigma=\{\alpha_j=\varepsilon_j-\varepsilon_{j+1}:1\leq j\leq \ell\}$$
respectively. The fundamental weights with respect to $\Sigma$ are $$\overline{\omega}_j=\frac{1}{\ell+1}\left[(\ell-j+1)(\varepsilon_1+\cdots+\varepsilon_j)-j(\varepsilon_{j+1}+\cdots+\varepsilon_{\ell+1})\right],\ j=1,...,\ell,$$
the highest root is
$$\tilde{\alpha}=\alpha_1+\cdots+\alpha_\ell=\varepsilon_1-\varepsilon_{\ell+1}$$
and $$\eta=\frac{1}{2}\sum\limits_{1\leq i<j\leq \ell+1}(\varepsilon_i-\varepsilon_j)=\frac{1}{2}\sum\limits_{k=1}^{\ell+1}(\ell-2(k-1))\varepsilon_k.$$ Then, we can compute
$$b=\langle\varepsilon_1-\varepsilon_{\ell+1},\varepsilon_1-\varepsilon_{\ell+1}\rangle+2\left\langle\varepsilon_1-\varepsilon_{\ell+1},\frac{1}{2}\sum\limits_{k=1}^{\ell+1}(\ell-2(k-1))\varepsilon_k\right\rangle=2\ell+2.$$
Consider a highest weight $\tilde{\nu}=\tilde{\nu}_1\overline{\omega}_1+\cdots+\tilde{\nu}_1\overline{\omega}_\ell$, then we have that
\begin{align*}\tilde{\nu}=&\frac{\left(\sum\limits_{j=1}^\ell(\ell+1-j)\tilde{\nu}_j\right)\varepsilon_1+\sum\limits_{k=2}^\ell\left(-\sum\limits_{j=1}^{k-1}j\tilde{\nu}_j+\sum\limits_{j=k}^\ell(\ell+1-j)\tilde{\nu}_j\right)\varepsilon_k-\left(\sum\limits_{j=1}^\ell j\tilde{\nu}_j\right)\varepsilon_{\ell+1}}{\ell+1}\\
=&\frac{\sum\limits_{k=1}^{\ell+1}c_k\varepsilon_k}{\ell+1},
\end{align*}
where $c_1:=\sum\limits_{j=1}^\ell(\ell+1-j)\tilde{\nu}_j,\ c_k:=-\sum\limits_{j=1}^{k-1}j\tilde{\nu}_j+\sum\limits_{j=k}^\ell(\ell+1-j)\tilde{\nu}_j,\ k=2,...,\ell,$ and $c_{\ell+1}:=-\sum\limits_{j=1}^\ell j\tilde{\nu}_j.$ Therefore \begin{align*}
&\tilde{\nu}+\eta=\sum\limits_{k=1}^{\ell+1}\left(\frac{2c_k}{2(\ell+1)}+\frac{(\ell-2(k-1))(\ell+1)}{2(\ell+1)}\right)\varepsilon_k,\\
&\langle\tilde{\nu}+\eta,\tilde{\nu}+\eta\rangle=\frac{1}{4(\ell+1)^2}\sum\limits_{k=1}^{\ell+1}\left(2c_k+(\ell-2(k-1))(\ell+1)\right)^2,\\
&\langle\eta,\eta\rangle=\frac{1}{4(\ell+1)^2}\sum\limits_{k=1}^{\ell+1}(\ell-2(k-1))^2(\ell+1)^2,
\end{align*}
and the eigenvalue $\lambda_{\tilde{\nu}}$ associated to $\tilde{\nu}$ is given by
\begin{align*}
    \lambda_{\tilde{\nu}}=&\frac{1}{b\gamma}\left[\langle\tilde{\nu}+\eta,\tilde{\nu}+\eta\rangle-\langle\eta,\eta\rangle\right]\\
    =&\frac{1}{8(\ell+1)^3\gamma}\left[\sum\limits_{k=1}^{\ell+1}\left(2c_k+(\ell-2(k-1))(\ell+1)\right)^2-\sum\limits_{k=1}^{\ell+1}(\ell-2(k-1))^2(\ell+1)^2\right]\\
    =&\frac{1}{8(\ell+1)^3\gamma}\left[\left(2\sum\limits_{j=1}^\ell(\ell+1-j)\tilde{\nu}_j+\ell(\ell+1)\right)^2\right.\\
    &\hspace{2cm}+\sum\limits_{k=2}^\ell\left(-2\sum\limits_{j=1}^{k-1}j\tilde{\nu}_j+2\sum\limits_{j=k}^\ell(\ell+1-j)\tilde{\nu}_j+(\ell-2(k-1))(\ell+1)\right)^2\\
    &\hspace{2cm}+\left.\left(2\sum\limits_{j=1}^\ell j\tilde{\nu}_j+\ell(\ell+1)\right)^2-\sum\limits_{k=1}^{\ell+1}(\ell-2(k-1))^2(\ell+1)^2\right].
\end{align*}
Let $\nu_j=\tilde{\nu}_j+1,\ j=1,...,\ell.$ The highest weight $\tilde{\nu}$ is uniquely determined by $\nu=(\nu_1,...,\nu_l)\in\mathbb{N}^\ell$. Let us write $\lambda_{\tilde{\nu}}$ in terms of $\nu:$ 
\begin{align*}
    2\sum\limits_{j=1}^\ell(\ell+1-j)\tilde{\nu}_j+\ell(\ell+1)=&2\sum\limits_{j=1}^\ell(\ell+1-j)(\nu_j-1)+\ell(\ell+1)\\
    =&2\sum\limits_{j=1}^\ell(\ell+1-j)\nu_j-2\sum\limits_{j=1}^\ell(\ell+1-j)+\ell(\ell+1)\\
    =&2\sum\limits_{j=1}^\ell(\ell+1-j)\nu_j-\ell(\ell+1)+\ell(\ell+1)\\
    =&2\sum\limits_{j=1}^\ell(\ell+1-j)\nu_j,
\end{align*}
\begin{align*}
    &-2\sum\limits_{j=1}^{k-1}j\tilde{\nu}_j+2\sum\limits_{j=k}^\ell(\ell+1-j)\tilde{\nu}_j+(\ell-2(k-1))(\ell+1)\\
    =&-2\sum\limits_{j=1}^{k-1}j(\nu_j-1)+2\sum\limits_{j=k}^\ell(\ell+1-j)(\nu_j-1)+(\ell-2(k-1))(\ell+1)\\
    =&-2\sum\limits_{j=1}^{k-1}j\nu_j+2\sum\limits_{j=k}^\ell(\ell+1-j)\nu_j+2\sum\limits_{j=1}^{k-1}j-2\sum\limits_{j=k}^\ell(\ell+1-j)\\
    &+(\ell-2(k-1))(\ell+1)\\
    =&-2\sum\limits_{j=1}^{k-1}j\nu_j+2\sum\limits_{j=k}^\ell(\ell+1-j)\nu_j+k(k-1)-(\ell+1-k)(\ell+2-k)\\
    &+(\ell-2(k-1))(\ell+1)\\
    =&-2\sum\limits_{j=1}^{k-1}j\nu_j+2\sum\limits_{j=k}^\ell(\ell+1-j)\nu_j-(\ell-2(k-1))(\ell+1)+(\ell-2(k-1))(\ell+1)\\
    =&-2\sum\limits_{j=1}^{k-1}j\nu_j+2\sum\limits_{j=k}^\ell(\ell+1-j)\nu_j,
\end{align*}
\begin{align*}
    2\sum\limits_{j=1}^\ell j\tilde{\nu}_j+\ell(\ell+1)=&2\sum\limits_{j=1}^\ell j(\nu_j-1)+\ell(\ell+1)\\
    =&2\sum\limits_{j=1}^\ell j\nu_j-2\sum\limits_{j=1}^\ell j+\ell(\ell+1)\\
    =&2\sum\limits_{j=1}^\ell j\nu_j-\ell(\ell+1)+\ell(\ell+1)\\
    =&2\sum\limits_{j=1}^\ell j\nu_j,\hspace{8.1cm}
\end{align*}
then
\begin{align*}
    \lambda_{\tilde{\nu}}=&\frac{1}{8(\ell+1)^3\gamma}\left[\left(2\sum\limits_{j=1}^\ell(\ell+1-j)\tilde{\nu}_j+\ell(\ell+1)\right)^2\right.\\
    &\hspace{2cm}+\sum\limits_{k=2}^\ell\left(-2\sum\limits_{j=1}^{k-1}j\tilde{\nu}_j+2\sum\limits_{j=k}^\ell(\ell+1-j)\tilde{\nu}_j+(\ell-2(k-1))(\ell+1)\right)^2\\
    &\hspace{2cm}+\left.\left(2\sum\limits_{j=1}^\ell j\tilde{\nu}_j+\ell(\ell+1)\right)^2-\sum\limits_{k=1}^{\ell+1}(\ell-2(k-1))^2(\ell+1)^2\right]\\
    =&\frac{1}{8(\ell+1)^3\gamma}\left[\left(2\sum\limits_{j=1}^\ell(\ell+1-j)\nu_j\right)^2\right.+\sum\limits_{k=2}^\ell\left(-2\sum\limits_{j=1}^{k-1}j\nu_j+2\sum\limits_{j=k}^\ell(\ell+1-j)\nu_j\right)^2\\
    &\hspace{2cm}+\left.\left(2\sum\limits_{j=1}^\ell j\nu_j\right)^2-\sum\limits_{k=1}^{\ell+1}(\ell-2(k-1))^2(\ell+1)^2\right]\\
    =&\frac{1}{2(\ell+1)^3\gamma}\left[\left(\sum\limits_{j=1}^\ell(\ell+1-j)\nu_j\right)^2\right.+\sum\limits_{k=2}^\ell\left(-\sum\limits_{j=1}^{k-1}j\nu_j+\sum\limits_{j=k}^\ell(\ell+1-j)\nu_j\right)^2\\
    &\hspace{2cm}+\left.\left(\sum\limits_{j=1}^\ell j\nu_j\right)^2-\frac{1}{4}\sum\limits_{k=1}^{\ell+1}(\ell-2(k-1))^2(\ell+1)^2\right].
\end{align*}
Since $\tilde{\nu}_j\geq 0$, then $\nu_j\geq 1.$ By setting $\nu=(\nu_1,...,\nu_\ell)$ and $\lambda_\nu=\lambda_{\tilde{\nu}}$ we obtain the result.
\end{proof}
\noindent $\bullet$ {\it Diagram} $B_\ell$ In this case, one has that $G\cong \textnormal{Spin}(2\ell+1).$ The spectrum of the Laplacian is computed in the following theorem. We exclude the case  $\ell=1$  since $\textnormal{Spin(3)}\cong \textnormal{SU}(2)\cong \mathbb{S}^3$ and the Strichartz estimates in the case of the spheres have been analysed in the work of Burq, Gerard, and Tzvetkov \cite{BurqGerardTzvetkov2004}.
\begin{theorem}\label{spectrum:theorem:Bl} The spectrum of the Laplacian operator on \textnormal{Spin(}$2\ell+1$\textnormal{) (}$\ell\geq 2$\textnormal{)} is given by
\begin{equation*}
\textnormal{Spec(Spin(}2\ell+1\textnormal{))}=\{\lambda_\nu:\nu=(\nu_1,...,\nu_\ell)\in\mathbb{N}^\ell\},
\end{equation*}
where
\begin{equation}\label{spectrum:Bl}
\lambda_\nu=\frac{1}{(16\ell-8)\gamma}\left(\sum\limits_{k=1}^{\ell-1}\left(2\sum\limits_{j=k}^{\ell-1}\nu_j+\nu_\ell\right)^2+\nu_\ell^2-\sum\limits_{k=1}^{\ell}\left(2(\ell-k)+1\right)^2\right).
\end{equation}
\end{theorem}
\begin{proof} The root system associated to Spin($2\ell+1$) is  $B_\ell$ which can be constructed by considering a $\langle\cdot,\cdot\rangle$-orthonormal basis $\{\varepsilon_1,...,\varepsilon_\ell\}$ on a $\ell$-dimensional Euclidean space. The positive roots are $\varepsilon_j,\ 1\leq j\leq \ell$ (the short roots) and $\varepsilon_i\pm\varepsilon_j,\ 1\leq i<j\leq \ell$ (the long roots). The set $\Sigma=\{\alpha_1=\varepsilon_1-\varepsilon_2,,...,\alpha_{\ell-1}=\varepsilon_{\ell-1}-\varepsilon_\ell,\alpha_\ell=\varepsilon_\ell\}$ is a system of simple positive roots and the corresponding fundamental weights are $\overline{\omega}_j=\varepsilon_1+\cdots+\varepsilon_j,$ $1\leq j<\ell$ and $\overline{\omega}_\ell=\frac{\varepsilon_1+\cdots+\varepsilon_\ell}{2}$. The highest long root is $\tilde{\alpha}=\alpha_1+2(\alpha_2+\cdots+\alpha_\ell)=\varepsilon_1+\varepsilon_2$ and $$\eta=\frac{1}{2}\left(\sum\limits_{j=1}^\ell\varepsilon_j+\sum\limits_{1\leq i<j\leq \ell}(\varepsilon_i+\varepsilon_j)+\sum\limits_{1\leq i<j\leq \ell}(\varepsilon_i-\varepsilon_j)\right)=\frac{1}{2}\left(\sum\limits_{k=1}^\ell(2(\ell-k)+1)\varepsilon_k\right),$$ thus
$$b=\langle\varepsilon_1+\varepsilon_2,\varepsilon_1+\varepsilon_2\rangle+2\left\langle\varepsilon_1+\varepsilon_2,\frac{1}{2}\left(\sum\limits_{k=1}^\ell(2(\ell-k)+1)\varepsilon_k\right)\right\rangle=4\ell-2.$$
Given a highest weight $$\tilde{\nu}=\tilde{\nu}_1\overline{\omega}_1+\cdots+\tilde{\nu}_\ell\overline{\omega}_\ell=\sum\limits_{k=1}^{\ell-1}\left[\sum\limits_{j=k}^{\ell-1}\tilde{\nu}_j+\frac{\tilde{\nu}_\ell}{2}\right]\varepsilon_k+\frac{\tilde{\nu}_\ell}{2}\varepsilon_\ell,\ \tilde{\nu}_j\in\mathbb{N}_0,\ j=1,...,\ell;$$ we have (by formula \eqref{eigenvalue:formula:3}) that the eigenvalue $\lambda_{\tilde{\nu}}$ associated to $\tilde{\nu}$ is 
\begin{equation*}
\lambda_{\tilde{\nu}}=\frac{1}{(4\ell-2)\gamma}\left[\langle\tilde{\nu}+\eta,\tilde{\nu}+\eta\rangle-\langle\eta,\eta\rangle\right]
\end{equation*}
but
$$\tilde{\nu}+\eta=\sum\limits_{k=1}^{\ell-1}\left\{\left[\sum\limits_{j=k}^{\ell-1}\tilde{\nu}_j+\frac{\tilde{\nu}_\ell}{2}\right]+\frac{(2(\ell-k)+1)}{2}\right\}\varepsilon_k+\frac{\tilde{\nu}_\ell+1}{2}\varepsilon_\ell$$
so that
\begin{align*}
\langle\tilde{\nu}+\eta,\tilde{\nu}+\eta\rangle=&\frac{1}{4}\left\{\sum\limits_{k=1}^{\ell-1}\left(2\sum\limits_{j=k}^{\ell-1}\tilde{\nu}_j+\tilde{\nu}_\ell+(2(\ell-k)+1)\right)^2+(\tilde{\nu}_\ell+1)^2\right\}.
\end{align*}
On the other hand
$$\langle\eta,\eta\rangle=\frac{1}{4}\sum\limits_{k=1}^{l}\left(2(l-k)+1\right)^2,$$
hence
\begin{align*}
\lambda_{\tilde{\nu}}=&\frac{1}{(16\ell-8)\gamma}\left(\sum\limits_{k=1}^{\ell-1}\left(2\sum\limits_{j=k}^{\ell-1}\tilde{\nu}_j+\tilde{\nu}_\ell+(2(\ell-k)+1)\right)^2+(\tilde{\nu}_\ell+1)^2\right.\\
&\hspace{2cm}-\left.\sum\limits_{k=1}^{\ell}\left(2(\ell-k)+1\right)^2\right).
\end{align*}
If $\nu_j:=\tilde{\nu}_j+1,\ j=1,...,\ell$, $\nu:=(\nu_1,...,\nu_\ell)\in\mathbb{N}^\ell,$ and $\lambda_\nu:=\lambda_{\tilde{\nu}}$ we have
\begin{align*}
\lambda_\nu=\lambda_{\tilde{\nu}}=&\frac{1}{(16\ell-8)\gamma}\left(\sum\limits_{k=1}^{\ell-1}\left(2\sum\limits_{j=k}^{\ell-1}\tilde{\nu}_j+\tilde{\nu}_\ell+(2(\ell-k)+1)\right)^2+(\tilde{\nu}_\ell+1)^2\right.\\
&\hspace{2cm}-\left.\sum\limits_{k=1}^{\ell}\left(2(\ell-k)+1\right)^2\right)\\
=&\frac{1}{(16\ell-8)\gamma}\left(\sum\limits_{k=1}^{\ell-1}\left(2\sum\limits_{j=k}^{\ell-1}(\nu_j-1)+(\nu_\ell-1)+(2(\ell-k)+1)\right)^2+\nu_\ell^2\right.\\
&\hspace{2cm}-\left.\sum\limits_{k=1}^{\ell}\left(2(\ell-k)+1\right)^2\right)\\
=&\frac{1}{(16\ell-8)\gamma}\left(\sum\limits_{k=1}^{\ell-1}\left(2\sum\limits_{j=k}^{\ell-1}\nu_j-2\sum\limits_{j=k}^{\ell-1}1+\nu_\ell-1+(2(\ell-k)+1)\right)^2+\nu_\ell^2\right.\\
&\hspace{2cm}-\left.\sum\limits_{k=1}^{\ell}\left(2(\ell-k)+1\right)^2\right)\\
=&\frac{1}{(16\ell-8)\gamma}\left(\sum\limits_{k=1}^{\ell-1}\left(2\sum\limits_{j=k}^{\ell-1}\nu_j+\nu_\ell-2(\ell-k)-1+(2(\ell-k)+1)\right)^2+\nu_\ell^2\right.\\
&\hspace{2cm}-\left.\sum\limits_{k=1}^{\ell}\left(2(\ell-k)+1\right)^2\right)\\
=&\frac{1}{(16\ell-8)\gamma}\left(\sum\limits_{k=1}^{\ell-1}\left(2\sum\limits_{j=k}^{\ell-1}\nu_j+\nu_\ell\right)^2+\nu_\ell^2-\sum\limits_{k=1}^{\ell}\left(2(\ell-k)+1\right)^2\right),
\end{align*}completing the proof.
\end{proof}
\noindent $\bullet$ {\it Diagram} $C_{\ell}$. In this case, one has that $G\cong \textnormal{Sp}(\ell).$  Observe that $\textnormal{Sp(1)}\cong\textnormal{SU(2)}$ was treated in Theorem \ref{spectrum:theorem:Al} and $\textnormal{Sp(2)}\cong\textnormal{SO(5)}$ has a universal cover diffeomorphic to $\textnormal{Spin(5)}$ whose spectrum was described in Theorem \ref{spectrum:theorem:Bl}. For $\ell\geq 3$, the spectrum of the Laplacian is computed in the following theorem. 
\begin{theorem}\label{spectrum:theorem:Cl} The spectrum of the Laplacian operator on \textnormal{Sp(}$\ell$\textnormal{), (}$\ell\geq 3$\textnormal{)} is given by
\begin{equation*}
\textnormal{Spec(Sp(}\ell\textnormal{))}=\{\lambda_\nu:\nu=(\nu_1,...,\nu_\ell)\in\mathbb{N}^\ell\},
\end{equation*}
where
\begin{equation}
\lambda_\nu=\frac{1}{4(\ell+1)\gamma}\left[\sum\limits_{k=1}^{\ell}\left(\sum\limits_{j=k}^\ell\nu_j\right)^2-\frac{\ell(\ell+1)(2\ell+1)}{6}\right].
\end{equation}
\end{theorem}
\begin{proof} The root system associated to the group Sp($\ell$) is $C_\ell$ which is obtained by considering a $\ell$-dimensional Euclidean space with an inner product $\langle\cdot,\cdot\rangle$ and a orthonormal basis $\{\varepsilon_1,...,\varepsilon_\ell\}$. The simple positive roots are $\alpha_j=\varepsilon_j-\varepsilon_{j+1},\ j=1,...,\ell-1$ and $\alpha_\ell=2\varepsilon_\ell.$ The positive roots are $2\varepsilon_i,\ i=1,...,\ell,\ \varepsilon_i\pm\varepsilon_j,\ 1\leq i<j\leq \ell$ and the fundamental weights are $\overline{w}_j=\varepsilon_1+\cdots+\varepsilon_j,\ j=1,...,\ell.$ In this case
\begin{align*}
    \tilde{\alpha}=&2(\alpha_1+\cdots+\alpha_{\ell-1})+\alpha_\ell=2\varepsilon_1\ \textnormal{and}\\
    \eta=&\frac{1}{2}\left(\sum\limits_{1\leq i<j\leq \ell}(\varepsilon_i+\varepsilon_j)+\sum\limits_{1\leq i<j\leq \ell}(\varepsilon_i-\varepsilon_j)+\sum\limits_{i=1}^\ell2\varepsilon_i\right)\\
    =&\sum\limits_{k=1}^{\ell}(\ell-k+1)\varepsilon_k,
\end{align*}
thus
\begin{align*}
    b=&\langle\tilde{\alpha},\tilde{\alpha}\rangle+2\left\langle\tilde{\alpha},\eta\right\rangle\\
    =&4\langle\varepsilon_1,\varepsilon_1\rangle+4\left\langle\varepsilon_1,\sum\limits_{k=1}^{\ell}(\ell-k+1)\varepsilon_k\right\rangle\\
    =&4(\ell+1).
\end{align*}
A highest weight has the form
$$\tilde{\nu}=\sum\limits_{k=1}^\ell\tilde{\nu}_k\overline{\omega}_k=\sum\limits_{k=1}^\ell\tilde{\nu}_k\sum\limits_{j=1}^k\varepsilon_j=\sum\limits_{k=1}^\ell\sum\limits_{j=1}^k\tilde{\nu}_k\varepsilon_j=\sum\limits_{k=1}^{\ell}\left[\sum\limits_{j=k}^\ell\tilde{\nu}_j\right]\varepsilon_k,$$ so $$\tilde{\nu}+\eta=\sum\limits_{k=1}^{\ell}\left[\sum\limits_{j=k}^\ell\tilde{\nu}_j\right]\varepsilon_k+\sum\limits_{k=1}^{\ell}(\ell-k+1)\varepsilon_k=\sum\limits_{k=1}^{\ell}\left[\sum\limits_{j=k}^\ell\tilde{\nu}_j+(\ell-k+1)\right]\varepsilon_k$$ which implies $$\langle\tilde{\nu}+\eta,\tilde{\nu}+\eta\rangle=\sum\limits_{k=1}^{\ell}\left(\sum\limits_{j=k}^\ell\tilde{\nu}_j+(\ell-k+1)\right)^2.$$
On the other hand $$\langle\eta,\eta\rangle=\sum\limits_{k=1}^{\ell}(\ell-k+1)^2=\frac{\ell(\ell+1)(2\ell+1)}{6}.$$
Hence, the eigenvalue $\lambda_{\tilde{\nu}}$ associated to $\tilde{\nu}$ is
\begin{align*}
    \lambda_{\tilde{\nu}}=&\frac{1}{b\gamma}\left[\langle\tilde{\nu}+\eta,\tilde{\nu}+\eta\rangle-\langle\eta,\eta\rangle\right]\\
    =&\frac{1}{4(\ell+1)\gamma}\left[\sum\limits_{k=1}^{\ell}\left(\sum\limits_{j=k}^\ell\tilde{\nu}_j+(\ell-k+1)\right)^2-\frac{\ell(\ell+1)(2\ell+1)}{6}\right]\\
    =&\frac{1}{4(\ell+1)\gamma}\left[\sum\limits_{k=1}^{\ell}\left(\sum\limits_{j=k}^\ell\tilde{\nu}_j+\sum\limits_{j=k}^{\ell}1\right)^2-\frac{\ell(\ell+1)(2\ell+1)}{6}\right]\\
    =&\frac{1}{4(\ell+1)\gamma}\left[\sum\limits_{k=1}^{\ell}\left(\sum\limits_{j=k}^\ell(\tilde{\nu}_j+1)\right)^2-\frac{\ell(\ell+1)(2\ell+1)}{6}\right]\\
    =&\frac{1}{4(\ell+1)\gamma}\left[\sum\limits_{k=1}^{\ell}\left(\sum\limits_{j=k}^\ell\nu_j\right)^2-\frac{\ell(\ell+1)(2\ell+1)}{6}\right],\\
\end{align*}
where $\nu_j=\tilde{\nu}_j+1,\ j=1,...,\ell$. By setting $\nu:=(\nu_1,...,\nu_\ell)$ and $\lambda_\nu:=\lambda_{\tilde{\nu}}$ we complete the proof.
\end{proof}
\noindent $\bullet$ {\it Diagram} $D_\ell.$ In this case, one has that $G\cong \textnormal{Spin}(2\ell).$  Note that $\textnormal{Spin}(2)\cong \mathbb{S}^1$ was treated in \cite{BurqGerardTzvetkov2004}, $\textnormal{Spin}(4)\cong \textnormal{SU}(2)\times \textnormal{SU}(2)$ is not simple and $\textnormal{Spin}(6)\cong \textnormal{SU}(4)$ was considered in Theorem \ref{spectrum:theorem:Al}. The spectrum of the Laplacian is computed in the following theorem for $\ell\geq 4.$
\begin{theorem}\label{spectrum:theorem:Dl} The spectrum of the Laplacian operator on \textnormal{Spin(}$2\ell$\textnormal{), (}$\ell\geq 4$\textnormal{)} is given by
\begin{equation*}
\textnormal{Spec(Spin(}2\ell\textnormal{))}=\{\lambda_\nu:\nu=(\nu_1,...,\nu_\ell)\in\mathbb{N}^\ell\},
\end{equation*}
where
\begin{align}
\lambda_\nu=&\frac{1}{(16\ell-16)\gamma}\left(\sum\limits_{k=1}^{\ell-2}\left(2\sum\limits_{j=k}^{\ell-2}\nu_j+\nu_{\ell-1}+\nu_\ell\right)^2+(\nu_\ell+\nu_{\ell-1})^2+(\nu_\ell-\nu_{\ell-1})^2\right.\nonumber\\
&\hspace{2cm}-\left.\sum\limits_{k=1}^{\ell}\left(2(\ell-k)\right)^2\right).
\end{align}
\end{theorem}
\begin{proof} The root system associated to Spin($2\ell$) is  $D_\ell$ which can be constructed by considering a $\langle\cdot,\cdot\rangle$-orthonormal basis $\{\varepsilon_1,...,\varepsilon_\ell\}$ on a $\ell$-dimensional Euclidean space. The positive roots are $\varepsilon_i\pm\varepsilon_j,\ 1\leq i<j\leq \ell$. The set $\Sigma=\{\alpha_1=\varepsilon_1-\varepsilon_2,,...,\alpha_{\ell-1}=\varepsilon_{\ell-1}-\varepsilon_\ell,\alpha_\ell=\varepsilon_{\ell-1}+\varepsilon_\ell\}$ is a system of simple positive roots and the corresponding fundamental weights are $\overline{\omega}_j=\varepsilon_1+\cdots+\varepsilon_j,$ $1\leq j<\ell-1,$ $\overline{\omega}_{\ell-1}=\frac{\varepsilon_1+\cdots+\varepsilon_{\ell-1}-\varepsilon_\ell}{2},$ and $\overline{\omega}_{\ell}=\frac{\varepsilon_1+\cdots+\varepsilon_{\ell-1}+\varepsilon_\ell}{2}$. The highest root is $$\tilde{\alpha}=\alpha_1+2(\alpha_2+\cdots+\alpha_{\ell-2})+\alpha_{\ell-1}+\alpha_\ell=\varepsilon_1+\varepsilon_2$$ and $$\eta=\frac{1}{2}\left(\sum\limits_{1\leq i<j\leq \ell}(\varepsilon_i+\varepsilon_j)+\sum\limits_{1\leq i<j\leq \ell}(\varepsilon_i-\varepsilon_j)\right)=\sum\limits_{k=1}^\ell(\ell-k)\varepsilon_k,$$ thus
$$b=\langle\varepsilon_1+\varepsilon_2,\varepsilon_1+\varepsilon_2\rangle+2\left\langle\varepsilon_1+\varepsilon_2,\sum\limits_{k=1}^\ell(\ell-k)\varepsilon_k\right\rangle=4\ell-4.$$
A highest weight has the form $$\tilde{\nu}=\tilde{\nu}_1\overline{\omega}_1+\cdots+\tilde{\nu}_\ell\overline{\omega}_\ell=\sum\limits_{k=1}^{\ell-2}\left[\sum\limits_{j=k}^{\ell-2}\tilde{\nu}_j+\frac{\tilde{\nu}_{\ell-1}+\tilde{\nu}_{\ell}}{2}\right]\varepsilon_k+\frac{\tilde{\nu}_{\ell-1}+\tilde{\nu}_\ell}{2}\varepsilon_{\ell-1}+\frac{\tilde{\nu}_\ell-\tilde{\nu}_{\ell-1}}{2}\varepsilon_\ell,$$ with $\tilde{\nu}_j\in\mathbb{N}_0,\ j=1,...,\ell.$ Let us compute $\langle\tilde{\nu}+\eta,\tilde{\nu}+\eta\rangle$ and $\langle\eta,\eta\rangle:$

$$\tilde{\nu}+\eta=\sum\limits_{k=1}^{\ell-2}\left\{\left[\sum\limits_{j=k}^{\ell-2}\tilde{\nu}_j+\frac{\tilde{\nu}_{\ell-1}+\tilde{\nu}_\ell}{2}\right]+(\ell-k)\right\}\varepsilon_k+\frac{\tilde{\nu}_{\ell-1}+\tilde{\nu}_{\ell}+2}{2}+\frac{\tilde{\nu}_\ell-\tilde{\nu}_{\ell-1}}{2}\varepsilon_\ell$$
so that
\begin{align*}
&\langle\tilde{\nu}+\eta,\tilde{\nu}+\eta\rangle\\
&=\frac{1}{4}\left\{\sum\limits_{k=1}^{\ell-2}\left(2\sum\limits_{j=k}^{\ell-2}\tilde{\nu}_j+\tilde{\nu}_{\ell-1}+\tilde{\nu}_\ell+2(\ell-k)\right)^2+(\tilde{\nu}_\ell+\tilde{\nu}_{\ell-1}+2)^2+(\tilde{\nu}_\ell-\tilde{\nu}_{\ell-1})^2\right\}
\end{align*}
and
$$\langle\eta,\eta\rangle=\sum\limits_{k=1}^{\ell}(\ell-k)^2=\frac{1}{4}\sum\limits_{k=1}^{\ell}(2(\ell-k))^2.$$
By formula \eqref{eigenvalue:formula:3}  we have
\begin{align*}
\lambda_{\tilde{\nu}}=&\frac{1}{(4\ell-4)\gamma}\left[\langle\tilde{\nu}+\eta,\tilde{\nu}+\eta\rangle-\langle\eta,\eta\rangle\right]\\
=&\frac{1}{(16\ell-16)\gamma}\left(\sum\limits_{k=1}^{\ell-2}\left(2\sum\limits_{j=k}^{\ell-2}\tilde{\nu}_j+\tilde{\nu}_{\ell-1}+\tilde{\nu}_\ell+2(\ell-k)\right)^2+(\tilde{\nu}_\ell+\tilde{\nu}_{\ell-1}+2)^2\right.\\
&\hspace{2cm}+(\tilde{\nu}_\ell-\tilde{\nu}_{\ell-1})^2-\left.\sum\limits_{k=1}^{\ell}\left(2(\ell-k)\right)^2\right).
\end{align*}
Set $\nu_j:=\tilde{\nu}_j+1,\ j=1,...,\ell$, $\nu:=(\nu_1,...,\nu_\ell)\in\mathbb{N}^\ell$, and $\lambda_\nu:=\lambda_{\tilde{\nu}}$, so that
\begin{align*}
\lambda_\nu=\lambda_{\tilde{\nu}}=&\frac{1}{(16\ell-16)\gamma}\left(\sum\limits_{k=1}^{\ell-2}\left(2\sum\limits_{j=k}^{\ell-2}(\nu_j-1)+(\nu_{\ell-1}-1)+(\nu_\ell-1)+2(\ell-k)\right)^2\right.\\
&\hspace{2cm}+(\nu_\ell-1+\nu_{\ell-1}-1+2)^2+(\nu_\ell-\nu_{\ell-1})^2-\left.\sum\limits_{k=1}^{\ell}\left(2(\ell-k)\right)^2\right)\\
=&\frac{1}{(16\ell-16)\gamma}\left(\sum\limits_{k=1}^{\ell-2}\left(2\sum\limits_{j=k}^{\ell-2}\nu_j-2\sum\limits_{j=k}^{\ell-2}1+\nu_{\ell-1}+\nu_\ell-2+2(\ell-k)\right)^2\right.\\
&\hspace{2cm}+(\nu_\ell+\nu_{\ell-1})^2+(\nu_\ell-\nu_{\ell-1})^2-\left.\sum\limits_{k=1}^{\ell}\left(2(\ell-k)\right)^2\right)\\
=&\frac{1}{(16\ell-16)\gamma}\left(\sum\limits_{k=1}^{\ell-2}\left(2\sum\limits_{j=k}^{\ell-2}\nu_j+\nu_{\ell-1}+\nu_\ell-2(\ell-k-1)-2+2(\ell-k)\right)^2\right.\\
&\hspace{2cm}+(\nu_\ell+\nu_{\ell-1})^2+(\nu_\ell-\nu_{\ell-1})^2-\left.\sum\limits_{k=1}^{\ell}\left(2(\ell-k)\right)^2\right)\\
=&\frac{1}{(16\ell-16)\gamma}\left(\sum\limits_{k=1}^{\ell-2}\left(2\sum\limits_{j=k}^{\ell-2}\nu_j+\nu_{\ell-1}+\nu_\ell\right)^2+(\nu_\ell+\nu_{\ell-1})^2+(\nu_\ell-\nu_{\ell-1})^2\right.\\
&\hspace{2cm}-\left.\sum\limits_{k=1}^{\ell}\left(2(\ell-k)\right)^2\right),
\end{align*}completing the proof.
\end{proof}
\noindent $\bullet$ {\it Diagram} $\textnormal{E}_8.$ Now, we consider the exceptional case of the compact Lie group $\textnormal{E}_8.$
\begin{theorem}\label{spectrum:theorem:E8}
    The spectrum of the Laplacian operator on $\textnormal{E}_8$ is given by $$\textnormal{Spec(E}_8\textnormal{)}=\{\lambda_\nu:\nu=(\nu_1,\nu_2,\nu_3,\nu_4,\nu_5,\nu_6,\nu_7,\nu_8)\in\mathbb{N}^8\}$$ where
    \begin{align}
        \lambda_\nu=&\frac{1}{240\gamma}\left[\left(\nu_2-\nu_3\right)^2+\left(\nu_2+\nu_3\right)^2+\left(\nu_2+\nu_3+2\nu_4\right)^2+\left(\nu_2+\nu_3+2\nu_4+2\nu_5\right)^2\right.\nonumber\\
    &\left.\hspace{1cm}+\left(\nu_2+\nu_3+2\nu_4+2\nu_5+2\nu_6\right)^2+\left(\nu_2+\nu_3+2\nu_4+2\nu_5+2\nu_6+2\nu_7\right)^2\right.\nonumber\\
    &\left.\hspace{1cm}+\left(\nu_2+\nu_3+2\nu_4+2\nu_5+2\nu_6+2\nu_7+2\nu_8\right)^2\right.\nonumber\\
    &\left.\hspace{1cm}+\left(4\nu_1+5\nu_2+7\nu_3+10\nu_4+8\nu_5+6\nu_6+4\nu_7+2\nu_8\right)^2-2480\right].
    \end{align}
\end{theorem}
\begin{proof} We can construct $\textnormal{E}_8$ on a 8-dimensional linear space with an inner product $\langle\cdot,\cdot\rangle$ and an orthonormal basis $\{\varepsilon_1,\cdots ,\varepsilon_8\}.$ The positive roots are given by
\begin{align*}
&\varepsilon_j\pm\varepsilon_i,\ 1\leq i< j\leq 8,\ \textnormal{and}\\
&\frac{1}{2}\left((-1)^{i_1}\varepsilon_1+\cdots+(-1)^{i_7}\varepsilon_7+\varepsilon_8\right),\ i_k\in\{0,1\},\ i_1+\cdots+i_7\equiv 0\ \textnormal{mod}\ 2
\end{align*}
and the simple roots are
\begin{align*}
    &\alpha_1=\frac{1}{2}\left(\varepsilon_1-\varepsilon_2-\varepsilon_3-\varepsilon_4-\varepsilon_5-\varepsilon_6-\varepsilon_7+\varepsilon_8\right),\\
    &\alpha_2=\varepsilon_1+\varepsilon_2\\
    &\alpha_j=\varepsilon_{j-1}-\varepsilon_{j-2},\ j=3,4,5,6,7,8.
\end{align*}
The corresponding fundamental weights are given by
\begin{align*}
    \overline{\omega}_1=&4\alpha_1+5\alpha_2+7\alpha_3+10\alpha_4+8\alpha_5+6\alpha_6+4\alpha_7+2\alpha_8=2\varepsilon_8\\
    \overline{\omega}_2=&5\alpha_1+8\alpha_2+10\alpha_3+15\alpha_4+12\alpha_5+9\alpha_6+6\alpha_7+3\alpha_8\\
    =&\frac{\varepsilon_1+\varepsilon_2+\varepsilon_3+\varepsilon_4+\varepsilon_5+\varepsilon_6+\varepsilon_7+5\varepsilon_8}{2}\\
    \overline{\omega}_3=&7\alpha_1+10\alpha_2+14\alpha_3+20\alpha_4+16\alpha_5+12\alpha_6+8\alpha_7+4\alpha_8\\
    =&\frac{-\varepsilon_1+\varepsilon_2+\varepsilon_3+\varepsilon_4+\varepsilon_5+\varepsilon_6+\varepsilon_7+7\varepsilon_8}{2}\\
    \overline{\omega}_4=&10\alpha_1+15\alpha_2+20\alpha_3+30\alpha_4+24\alpha_5+18\alpha_6+12\alpha_7+6\alpha_8\\
    =&\varepsilon_3+\varepsilon_4+\varepsilon_5+\varepsilon_6+\varepsilon_7+5\varepsilon_8\\
    \overline{\omega}_5=&8\alpha_1+12\alpha_2+16\alpha_3+24\alpha_4+20\alpha_5+15\alpha_6+10\alpha_7+5\alpha_8\\
    =&\varepsilon_4+\varepsilon_5+\varepsilon_6+\varepsilon_7+4\varepsilon_8\\
    \overline{\omega}_6=&6\alpha_1+9\alpha_2+12\alpha_3+18\alpha_4+15\alpha_5+12\alpha_6+8\alpha_7+4\alpha_8\\
    =&\varepsilon_5+\varepsilon_6+\varepsilon_7+3\varepsilon_8\\
    \overline{\omega}_7=&4\alpha_1+6\alpha_2+8\alpha_3+12\alpha_4+10\alpha_5+8\alpha_6+6\alpha_7+3\alpha_8\\
    =&\varepsilon_6+\varepsilon_7+2\varepsilon_8\\
    \overline{\omega}_8=&2\alpha_1+3\alpha_2+4\alpha_3+6\alpha_4+5\alpha_5+4\alpha_6+3\alpha_7+2\alpha_8\\
    =&\varepsilon_7+\varepsilon_8,
\end{align*}
the highest long root is $\tilde{\alpha}=\varepsilon_7+\varepsilon_8$ and
\begin{align*}
\eta=&\frac{1}{2}\sum\limits_{\alpha\in\Pi^+}\alpha\\
=&\frac{1}{2}\left(\sum\limits_{1\leq i<j\leq 8}(\varepsilon_j+\varepsilon_i)+\sum\limits_{1\leq i<j\leq 8}(\varepsilon_j-\varepsilon_i)+\frac{1}{2}\left[\sum\limits_{\begin{subarray}{c}
i_1,...,i_7=0\\
i_1+\cdots+i_7\equiv0\ \textnormal{mod}\ 2
\end{subarray}}^1\left(\sum\limits_{k=1}^7(-1)^{i_k}\varepsilon_k+\varepsilon_8\right)\right]\right)\\
=&\frac{1}{2}\left(\sum\limits_{1\leq i<j\leq 8}2\varepsilon_j+32\varepsilon_8\right)\\
=&\sum\limits_{k=2}^7(k-1)\varepsilon_k+23\varepsilon_8,
\end{align*}
so
\begin{align*}
    b=\langle\tilde{\alpha},\tilde{\alpha}\rangle+2\langle\tilde{\alpha},\eta\rangle=2+2(6+23)=60.
\end{align*}
Let $\tilde{\nu}=\sum\limits_{k=1}^8\tilde{\nu}_j\overline{\omega}_j$ be a highest weight, then
\begin{align*}
    \tilde{\nu}+\eta=&\left(\frac{\tilde{\nu}_2-\tilde{\nu}_3}{2}\right)\varepsilon_1+\left(\frac{\tilde{\nu}_2+\tilde{\nu}_3+2}{2}\right)\varepsilon_2+\left(\frac{\tilde{\nu}_2+\tilde{\nu}_3+2\tilde{\nu}_4+4}{2}\right)\varepsilon_3\\
    &+\left(\frac{\tilde{\nu}_2+\tilde{\nu}_3+2\tilde{\nu}_4+2\tilde{\nu}_5+6}{2}\right)\varepsilon_4+\left(\frac{\tilde{\nu}_2+\tilde{\nu}_3+2\tilde{\nu}_4+2\tilde{\nu}_5+2\tilde{\nu}_6+8}{2}\right)\varepsilon_5\\
    &+\left(\frac{\tilde{\nu}_2+\tilde{\nu}_3+2\tilde{\nu}_4+2\tilde{\nu}_5+2\tilde{\nu}_6+2\tilde{\nu}_7+10}{2}\right)\varepsilon_6\\
    &+\left(\frac{\tilde{\nu}_2+\tilde{\nu}_3+2\tilde{\nu}_4+2\tilde{\nu}_5+2\tilde{\nu}_6+2\tilde{\nu}_7+2\tilde{\nu}_8+12}{2}\right)\varepsilon_7\\
    &+\left(\frac{4\tilde{\nu}_1+5\tilde{\nu}_2+7\tilde{\nu}_3+10\tilde{\nu}_4+8\tilde{\nu}_5+6\tilde{\nu}_6+4\tilde{\nu}_7+2\tilde{\nu}_8+46}{2}\right)\varepsilon_8\\
    =&\left(\frac{\nu_2-\nu_3}{2}\right)\varepsilon_1+\left(\frac{\nu_2+\nu_3}{2}\right)\varepsilon_2+\left(\frac{\nu_2+\nu_3+2\nu_4}{2}\right)\varepsilon_3\\
    &+\left(\frac{\nu_2+\nu_3+2\nu_4+2\nu_5}{2}\right)\varepsilon_4+\left(\frac{\nu_2+\nu_3+2\nu_4+2\nu_5+2\nu_6}{2}\right)\varepsilon_5\\
    &+\left(\frac{\nu_2+\nu_3+2\nu_4+2\nu_5+2\nu_6+2\nu_7}{2}\right)\varepsilon_6\\
    &+\left(\frac{\nu_2+\nu_3+2\nu_4+2\nu_5+2\nu_6+2\nu_7+2\nu_8}{2}\right)\varepsilon_7\\
    &+\left(\frac{4\nu_1+5\nu_2+7\nu_3+10\nu_4+8\nu_5+6\nu_6+4\nu_7+2\nu_8}{2}\right)\varepsilon_8,\\
\end{align*}
with $\nu_j=\tilde{\nu}_j+1,\ j=1,2,3,4,5,6,7,8.$ By setting $\nu:=(\nu_1,...,\nu_8)$ and $\lambda_\nu:=\lambda_{\tilde{\nu}}$ we have that
\begin{align*}
    \lambda_{\nu}=&\frac{1}{b\gamma}\left[\langle\tilde{\nu}+\eta,\tilde{\nu}+
    \eta\rangle-\langle\eta,\eta\rangle\right]\\
    =&\frac{1}{60\gamma}\left[\left(\frac{\nu_2-\nu_3}{2}\right)^2+\left(\frac{\nu_2+\nu_3}{2}\right)^2+\left(\frac{\nu_2+\nu_3+2\nu_4}{2}\right)^2+\left(\frac{\nu_2+\nu_3+2\nu_4+2\nu_5}{2}\right)^2\right.\\
    &\left.\hspace{1cm}+\left(\frac{\nu_2+\nu_3+2\nu_4+2\nu_5+2\nu_6}{2}\right)^2+\left(\frac{\nu_2+\nu_3+2\nu_4+2\nu_5+2\nu_6+2\nu_7}{2}\right)^2\right.\\
    &\left.\hspace{1cm}+\left(\frac{\nu_2+\nu_3+2\nu_4+2\nu_5+2\nu_6+2\nu_7+2\nu_8}{2}\right)^2\right.\\
    &\left.\hspace{1cm}+\left(\frac{4\nu_1+5\nu_2+7\nu_3+10\nu_4+8\nu_5+6\nu_6+4\nu_7+2\nu_8}{2}\right)^2\right.\\
    &\left.\hspace{1cm}-\left(\sum\limits_{k=2}^7(k-1)^2+23^2\right)\right]\\
    =&\frac{1}{240\gamma}\left[\left(\nu_2-\nu_3\right)^2+\left(\nu_2+\nu_3\right)^2+\left(\nu_2+\nu_3+2\nu_4\right)^2+\left(\nu_2+\nu_3+2\nu_4+2\nu_5\right)^2\right.\\
    &\left.\hspace{1cm}+\left(\nu_2+\nu_3+2\nu_4+2\nu_5+2\nu_6\right)^2+\left(\nu_2+\nu_3+2\nu_4+2\nu_5+2\nu_6+2\nu_7\right)^2\right.\\
    &\left.\hspace{1cm}+\left(\nu_2+\nu_3+2\nu_4+2\nu_5+2\nu_6+2\nu_7+2\nu_8\right)^2\right.\\
    &\left.\hspace{1cm}+\left(4\nu_1+5\nu_2+7\nu_3+10\nu_4+8\nu_5+6\nu_6+4\nu_7+2\nu_8\right)^2-2480\right],
\end{align*}
as we wanted to prove.
\end{proof}
\noindent $\bullet$ {\it Diagram} $\textnormal{F}_4.$ Now, we consider the exceptional case of the compact Lie group $\textnormal{F}_4.$
\begin{theorem}\label{spectrum:theorem:F4}
    The spectrum of the Laplacian operator on $\textnormal{F}_4$ is given by $$\textnormal{Spec(F}_4\textnormal{)}=\{\lambda_\nu:\nu=(\nu_1,\nu_2,\nu_3,\nu_4)\in\mathbb{N}^4\},$$ where
    \begin{equation}
        \lambda_\nu=\frac{1}{72\gamma}\left[(2\nu_1+4\nu_2+3\nu_3+2\nu_4)^2+(2\nu_1+2\nu_2+\nu_3)^2+(2\nu_2+\nu_3)^2+\nu_4^2-156\right].
    \end{equation}
\end{theorem}
\begin{proof} Let $\{\varepsilon_1,\varepsilon_2,\varepsilon_3,\varepsilon_4\}$ an orthonormal basis of a 4-dimensional vector space endowed with an inner product $\langle\cdot,\cdot\rangle.$ The simple roots are $\alpha_1=\varepsilon_2-\varepsilon_3,\ \alpha_2=\varepsilon_3-\varepsilon_4,\ \alpha_3=\varepsilon_4,$ and $\alpha_4=\frac{1}{2}(\varepsilon_1-\varepsilon_2-\varepsilon_3-\varepsilon_4)$ and the positive roots are $\varepsilon_i,\ 1\leq i\leq 4,\ \varepsilon_i\pm\varepsilon_j,\ 1\leq i<j\leq 4,$ and $\frac{1}{2}(\varepsilon_1\pm\varepsilon_2\pm\varepsilon_3\pm\varepsilon_4).$ The fundamental weights with respect to $\Sigma=\{\alpha_1,\alpha_2,\alpha_3,\alpha_4\}$ are
$$\overline{\omega}_1=\varepsilon_1+\varepsilon_2,\ \overline{\omega}_2=2\varepsilon_1+\varepsilon_2+\varepsilon_3,\ \overline{\omega}_3=\frac{3\varepsilon_1+\varepsilon_2+\varepsilon_3+\varepsilon_4}{2},\ \overline{\omega}_4=\varepsilon_1,$$ the highest long root is $$\tilde{\alpha}=2\alpha_1+3\alpha_2+4\alpha_3+2\alpha_4=\varepsilon_1+\varepsilon_2,$$ and
\begin{align*}\eta=&\frac{1}{2}\sum\limits_{\alpha\in\Pi^+}\alpha\\
=&\frac{1}{2}\left(\sum\limits_{i=1}^4\varepsilon_i+\sum\limits_{1\leq i<j\leq 4}(\varepsilon_i+\varepsilon_j)+\sum\limits_{1\leq i<j\leq 4}(\varepsilon_i-\varepsilon_j)\right.\\
&\left.\hspace{0.5cm}+\frac{1}{2}\sum\limits_{i,j,k=0}^1(\varepsilon_1+(-1)^i\varepsilon_2+(-1)^j\varepsilon_3+(-1)^k\varepsilon_4\right)\\
=&\frac{11}{2}\varepsilon_1+\frac{5}{2}\varepsilon_2+\frac{3}{2}\varepsilon_3+\frac{1}{2}\varepsilon_4,\\
\Longrightarrow b=&\langle\tilde{\alpha},\tilde{\alpha}\rangle+2\langle\tilde{\alpha},\eta\rangle= 2+2\left(\frac{11}{2}+\frac{5}{2}\right)=18.\\
\end{align*}
For a highest weight \begin{align*}\tilde{\nu}=&\ \tilde{\nu}_1\overline{\omega}_1+\tilde{\nu}_2\overline{\omega}_2+\tilde{\nu}_3\overline{\omega}_3+\tilde{\nu}_4\overline{\omega}_4\\
=&\left(\tilde{\nu}_1+2\tilde{\nu}_2+\frac{3\tilde{\nu}_3}{2}+\tilde{\nu}_4\right)\varepsilon_1+\left(\tilde{\nu}_1+\tilde{\nu}_2+\frac{\tilde{\nu}_3}{2}\right)\varepsilon_2+\left(\tilde{\nu}_2+\frac{\tilde{\nu}_3}{2}\right)\varepsilon_3+\frac{\tilde{\nu}_4}{2}\varepsilon_4\\
=&\left(\frac{2\tilde{\nu}_1+4\tilde{\nu}_2+3\tilde{\nu}_3+2\tilde{\nu}_4}{2}\right)\varepsilon_1+\left(\frac{2\tilde{\nu}_1+2\tilde{\nu}_2+\tilde{\nu}_3}{2}\right)\varepsilon_2+\left(\frac{2\tilde{\nu}_2+\tilde{\nu}_3}{2}\right)\varepsilon_3+\frac{\tilde{\nu}_4}{2}\varepsilon_4
\end{align*}
we have that
\begin{align*}
    \tilde{\nu}+\eta=&\left(\frac{2\tilde{\nu}_1+4\tilde{\nu}_2+3\tilde{\nu}_3+2\tilde{\nu}_4+11}{2}\right)\varepsilon_1+\left(\frac{2\tilde{\nu}_1+2\tilde{\nu}_2+\tilde{\nu}_3+5}{2}\right)\varepsilon_2\\
    &+\left(\frac{2\tilde{\nu}_2+\tilde{\nu}_3+3}{2}\right)\varepsilon_3+\left(\frac{\tilde{\nu}_4+1}{2}\right)\varepsilon_4\\
    =&\left(\frac{2(\tilde{\nu}_1+1)+4(\tilde{\nu}_2+1)+3(\tilde{\nu}_3+1)+2(\tilde{\nu}_4+1)}{2}\right)\varepsilon_1\\
    &+\left(\frac{2(\tilde{\nu}_1+1)+2(\tilde{\nu}_2+1)+(\tilde{\nu}_3+1)}{2}\right)\varepsilon_2+\left(\frac{2(\tilde{\nu}_2+1)+(\tilde{\nu}_3+1)}{2}\right)\varepsilon_3\\
    &+\left(\frac{\tilde{\nu}_4+1}{2}\right)\varepsilon_4\\
    =&\left(\frac{2\nu_1+4\nu_2+3\nu_3+2\nu_4}{2}\right)\varepsilon_1+\left(\frac{2\nu_1+2\nu_2+\nu_3}{2}\right)\varepsilon_2+\left(\frac{2\nu_2+\nu_3}{2}\right)\varepsilon_3+\frac{\nu_4}{2}\varepsilon_4
\end{align*}
where $\nu_j=\tilde{\nu}_j+1,\ j=1,2,3,4.$ Hence, $\nu:=(\nu_1,\nu_2,\nu_3,\nu_4)\in\mathbb{N}^4$ and the corresponding eigenvalue $\lambda_{\nu}:=\lambda_{\tilde{\nu}}$ is
\begin{align*}
\lambda_{\nu}=&\frac{1}{b\gamma}\left[\langle\tilde{\nu}+\eta,\tilde{\nu}+\eta\rangle-\langle\eta,\eta\rangle\right]\\
=&\frac{1}{18\gamma}\left[\left(\frac{2\nu_1+4\nu_2+3\nu_3+2\nu_4}{2}\right)^2+\left(\frac{2\nu_1+2\nu_2+\nu_3}{2}\right)^2+\left(\frac{2\nu_2+\nu_3}{2}\right)^2+\left(\frac{\nu_4}{2}\right)^2\right.\\
&\left.\hspace{0.8cm}-\left(\left(\frac{11}{2}\right)^2+\left(\frac{5}{2}\right)^2+\left(\frac{3}{2}\right)^2+\left(\frac{1}{2}\right)^2\right)\right]\\
=&\frac{1}{72\gamma}\left[(2\nu_1+4\nu_2+3\nu_3+2\nu_4)^2+(2\nu_1+2\nu_2+\nu_3)^2+(2\nu_2+\nu_3)^2+\nu_4^2-156\right],
\end{align*}
which completes the proof.
\end{proof}
\noindent $\bullet$ {\it Diagram} $\textnormal{G}_2.$ Now, we consider the exceptional case of the compact Lie group $\textnormal{G}_2.$
\begin{theorem}\label{spectrum:theorem:G2} The spectrum of the Laplacian operator on $\textnormal{G}_2$ is given by
\begin{equation*}
\textnormal{Spec(G}_2\textnormal{)}=\{\lambda_\nu:\nu=(\nu_1,\nu_2)\in\mathbb{N}\times\mathbb{N}\},
\end{equation*}
where
\begin{equation}
    \lambda_\nu=\frac{1}{24\gamma}\left[(\nu_1+\nu_2)^2+(\nu_1+2\nu_2)^2+\nu_2^2-14\right].
\end{equation}
\end{theorem}
\begin{proof} The root system $\textnormal{G}_2$ is constructed in a  3-dimensional space by considering the subspace $\{\varepsilon_1+\varepsilon_2+\varepsilon_3\}^{\perp}$ with respect to an inner product $\langle\cdot,\cdot\rangle$ and a  $\langle\cdot,\cdot\rangle$-orthonormal basis $\{\varepsilon_1,\varepsilon_2,\varepsilon_3\}.$  The sets of positive roots and simple positive roots are $$\Pi^+=\{\varepsilon_1-\varepsilon_2,\varepsilon_3-\varepsilon_2,\varepsilon_3-\varepsilon_1,\varepsilon_2+\varepsilon_3-2\varepsilon_1,\varepsilon_1+\varepsilon_3-2\varepsilon_2,2\varepsilon_3-\varepsilon_1-\varepsilon_2\}$$
\textnormal{and} 
$$\Sigma=\{\alpha_1=\varepsilon_1-\varepsilon_2,\alpha_2=\varepsilon_2+\varepsilon_3-2\varepsilon_1\}$$
respectively. The fundamental weights with respect to $\Sigma$ are $$\overline{\omega}_1=2\alpha_1+\alpha_2=\varepsilon_3-\varepsilon_2,\ \overline{\omega}_2=3\alpha_1+2\alpha_2=2\varepsilon_3-\varepsilon_1-\varepsilon_2,$$
the highest long root is
$$\tilde{\alpha}=3\alpha_1+2\alpha_2=2\varepsilon_3-\varepsilon_1-\varepsilon_2$$
and \begin{align*}\eta=&\frac{1}{2}(\varepsilon_1-\varepsilon_2+\varepsilon_3-\varepsilon_2+\varepsilon_3-\varepsilon_1+\varepsilon_2+\varepsilon_3-2\varepsilon_1+\varepsilon_1+\varepsilon_3-2\varepsilon_2+2\varepsilon_3-\varepsilon_1-\varepsilon_2)\\
=&-\varepsilon_1-2\varepsilon_2+3\varepsilon_3
\end{align*} so that
$$b=\langle\tilde{\alpha},\tilde{\alpha}\rangle+2\left\langle\tilde{\alpha},\eta\right\rangle=6+2(9)=24.$$
A highest weight $\tilde{\nu}=\tilde{\nu}_1\overline{\omega}_1+\tilde{\nu}_2\overline{\omega}_2$ is written in the basis $\{\varepsilon_1,\varepsilon_2,\varepsilon_3\}$ as
$$\tilde{\nu}=\tilde{\nu}_1(\varepsilon_3-\varepsilon_2)+\tilde{\nu}_2(2\varepsilon_3-\varepsilon_1-\varepsilon_2)=-\tilde{\nu}_2\varepsilon_1-(\tilde{\nu}_1+\tilde{\nu}_2)\varepsilon_2+(\tilde{\nu}_1+2\tilde{\nu}_2)\varepsilon_3,$$ hence $$\tilde{\nu}+\eta=-(\tilde{\nu}_2+1)\varepsilon_1-(\tilde{\nu}_1+\tilde{\nu}_2+2)\varepsilon_2+(\tilde{\nu}_1+2\tilde{\nu}_2+3)\varepsilon_3.$$
We can now find the eigenvalue $\lambda_{\tilde{\nu}}$ associated to $\tilde{\nu}:$
\begin{align*}
\lambda_{\tilde{\nu}}=&\frac{1}{b\gamma}\left[\langle\tilde{\nu}+\eta,\tilde{\nu}+\eta\rangle-\langle\eta,\eta\rangle\right]\\
=&\frac{1}{24\gamma}\left[(\tilde{\nu}_2+1)^2+(\tilde{\nu}_1+\tilde{\nu}_2+2)^2+(\tilde{\nu}_1+2\tilde{\nu}_2+3)^2-((-1)^2+(-2)^2+3^2)\right]\\
=&\frac{1}{24\gamma}\left[\nu_2^2+(\nu_1+\nu_2)^2+(\nu_1+2\nu_2)^2-14\right],
\end{align*}
with $\nu_j:=\tilde{\nu}_j,\ j=1,2.$ As before, we define $\nu:=(\nu_1,\nu_2)\in\mathbb{N}^2$ and $\lambda_{\nu}:=\lambda_{\tilde{\nu}}$ and obtain the result.
\end{proof}
The following important theorem summarises all the previous results.
\begin{theorem}\label{general:spectrum:theorem} Let $G$ be a connected compact simple Lie group of rank $\ell\geq 2$ and $\tilde{G}$ its universal cover. Then there exist $b,R_0\in\mathbb{N}$  and $\Omega\subseteq \mathbb{N}^\ell$ such that the spectrum of the Laplace operator on $G$ is
\begin{equation}\label{general:spectrum:set}
    \textnormal{Spec($G$)}=\{\lambda_\nu:\nu=(\nu_1,...,\nu_\ell)\in \Omega\},
\end{equation}
where $\lambda_\nu$ is given by the formula
\begin{align}
    \lambda_\nu=\frac{1}{b\gamma}(n_1^2+\cdots+n_m^2-R_0),
\end{align}
with $$m=\left\{\begin{array}{lll}
\ell+1, &&\textnormal{if}\ \tilde{G}\cong\textnormal{SU($\ell+1$)}\ \textnormal{or}\ \tilde{G}\cong G_2,\\
\ell, &&\textnormal{otherwise},
\end{array}\right.$$
and $$[n_1\ \cdots\ n_m]^T=A[\nu_1\ \cdots\ \nu_\ell]^T$$
for some integer $m\times \ell$ matrix $A$ with maximal rank equal to $\ell.$ 
\end{theorem}
\begin{proof} By Remark \ref{nonsimply:connected} we may assume that $G$ is simply-connected. As a consequence of Theorems \ref{spectrum:theorem:Al}, \ref{spectrum:theorem:Bl}, \ref{spectrum:theorem:Cl}, \ref{spectrum:theorem:Dl}, \ref{spectrum:theorem:E8}, \ref{spectrum:theorem:F4}, and \ref{spectrum:theorem:G2} we can conclude that the spectrum of the Laplace operator on $G$ has, in fact, the form of \eqref{general:spectrum:set}, where $b$ and $R_0$ are as in the following table
\begin{table}[H]
\centering
    \begin{tabular}{|c|c|c|}\hline
        Type of the Lie algebra &  $b$ & $R_0$\\\hline
        $A_\ell$ & $2(\ell+1)^3$ & $\frac{1}{4}\sum\limits_{k=1}^{\ell+1}(\ell-2(k-1))^2(\ell+1)^2$\\\hline
        $B_\ell$ & $16\ell-8$ & $\sum\limits_{k=1}^{\ell}\left(2(\ell-k)+1\right)^2$\\\hline
        $C_\ell$ & $4(\ell+1)$ & $\frac{\ell(\ell+1)(2\ell+1)}{6}$\\\hline
        $D_\ell$ & $16\ell-16$ & $\sum\limits_{k=1}^{\ell}\left(2(\ell-k)\right)^2$\\\hline
        $\textnormal{E}_8$ & 240 & 2480\\\hline
        $\textnormal{F}_4$ & 72 & 156\\\hline
        $\textnormal{G}_2$ & 24 & 14\\\hline
    \end{tabular}
    \label{table:1}
    \caption{Values of $b$ and $R_0.$}
\end{table}
 and in the formula $[n_1\ \cdots\ n_m]^T=A[\nu_1\ \cdots\ \nu_\ell]^T$ we have that
\begin{align*}
    &A=\begin{bmatrix}
    \ell & \ell-1 & \ell-2 & \ell-3 &\cdots & 2 &   1 \\
    -1 & \ell-1 & \ell-2 & \ell-3  &\cdots & 2 &   1 \\
    -1 & -2 & \ell-2 & \ell-3 &\cdots & 2 &   1 \\
   -1 & -2 & -3 &\ell-3 &\cdots  & 2 &   1 \\
   \cdots  &\cdots  & \cdots  &\cdots  &\cdots  & \cdots  &   \cdots  \\
   -1 & -2 & -3 &-4 &\cdots  & -(\ell-1) &   1 \\
   1 & 2 & 3 &4 &\cdots  & \ell-1 &   \ell \\
\end{bmatrix}_{(\ell+1)\times \ell},\ \textnormal{if $G$ is of type $A_\ell$},\\
    &A=
\begin{bmatrix}
    2 & 2 & \cdots  &   2 & 1 \\
    0 & 2 & \cdots  &   2 & 1  \\
    0 & 0 & \cdots  &   2 & 1  \\
    0 & 0 & \cdots  &   2 & 1  \\
   0 & 0 & \cdots  &   0 & 1  \\
\end{bmatrix}_{\ell\times \ell},\ \textnormal{if $G$ is of type $B_\ell$,}\\
    &A=
\begin{bmatrix}
    1 & 1 & 1&  1 &\cdots & 1 &   1 \\
    0 & 1 & 1 & 1 & \cdots & 1 &   1 \\
   0 & 0 & 1 & 1  & \cdots & 1 &   1 \\
   \cdots  &\cdots  & \cdots  &\cdots  &\cdots  & \cdots  &   \cdots  \\
   0& 0 & 0 & 0  & \cdots  & 1 &   1 \\
   0 & 0 & 0 & 0 &\cdots  & 0 &   1 \\
\end{bmatrix}_{\ell\times \ell},\ \textnormal{if $G$ is of type $C_\ell$,}\\
    &A=
\begin{bmatrix}
    2 & 2 & 2& \cdots  &   2 & 1 &1 \\
    0 & 2 & 2& \cdots  &   2 & 1 &1 \\
   0 & 0 & 2& \cdots  &   2 & 1 &1 \\
   \cdots  & \cdots  & \cdots & \cdots  &  \cdots  & \cdots  &\cdots  \\
   0 & 0 & 0& \cdots  &   2 & 1 &1 \\
   0 & 0 & 0& \cdots  &   0 & 1 &1 \\
   0 & 0 & 0& \cdots  &   0 & -1 & 1 \\
\end{bmatrix}_{\ell\times \ell},\ \textnormal{if $G$ is of type $D_\ell$},\\
    &A=
\begin{bmatrix}
    0 & 1 & -1 & 0 & 0 & 0 & 0 & 0\\
    0 & 1 & 1  & 0 & 0 & 0 & 0 & 0\\
    0 & 1 & 1  & 2 & 0 & 0 & 0 & 0\\
    0 & 1 & 1  & 2 & 2 & 0 & 0 & 0\\
    0 & 1 & 1  & 2 & 2 & 2 & 0 & 0\\
    0 & 1 & 1  & 2 & 2 & 2 & 2 & 0\\
    0 & 1 & 1  & 2 & 2 & 2 & 2 & 2\\
    4 & 5 & 7  & 10 & 8 & 6 & 4 & 2\\
\end{bmatrix},\ \textnormal{if $G\cong E_8$},\\
    &A=
\begin{bmatrix}
    2 & 4 & 3 &   2 \\
    2 & 2 & 1 &   0 \\
    0 & 2 & 1 & 0\\
   0 & 0 & 1&0\\
\end{bmatrix},\ \textnormal{if $G\cong$F}_4,\ \textnormal{and}\ A=
\begin{bmatrix}
    1& 1 \\
    1 & 2 \\
    0 & 1 \\
\end{bmatrix},\ \textnormal{if $G\cong G_2.$}
\end{align*}
    It is easy to see that all these matrices have all their columns linearly independent which completes the proof.
\end{proof}
\subsection{Tools from analytic number theory}
In this section we record some results from the analytic number theory applied in this manuscript related to the Waring problem.  For our further analysis, we are interested in the problem of analysing the number of solutions in positive integers $n_1,n_2,\cdots , n_{s}$ of the equation
\begin{equation}
    n_1^k+\cdots n_s^k=R,\quad R\in \mathbb{N},
\end{equation} which is  denoted by $r_{s,k}(R).$  This problem has a long history. The case $k=2,$ that is counting representations of a number by sums of $s$-squares is well understood  in view of  the circle method due to Hardy, Littlewood and Ramanujan \cite{Hardy1920,HardyLittlewood1920,HardyLittlewood1924,HardyLitlewood1925,HardyRamanujan1917:,HardyRamanujan1917,HardyRamanujan1918,HardyRamanujan1919,Ramanujan1917,Ramanujan1918}, with other developements e.g. by  Bateman \cite{Bateman1951} and Kloosterman \cite{Kloosterman1926}.  The case $k\geq 3$ is extremely exciting.  With earlier fundamental contributions from Hardy and Littlewood and further developments by Vinogradov \cite{Vinogradov1954} the problem  of estimating $r_{s,k}(N)$ when $k\geq 3,$ is related to both fields, analytic number theory and harmonic analysis. It is related to the Vinogradov mean value theorem proved in 2015 by Bourgain, Demeter and Guth \cite{BourgainDemeterGuth2016}, using the decoupling theorem of Bourgain and Demeter \cite{BourgainDemeter2015}. The work \cite{BourgainDemeter2015} and further developments have shown deep connections between analytic number theory and other fundamental problems in Fourier analysis, like the Stein restriction conjecture and the Kakeya problem.   We refer to Pierce \cite{Pierce2019} for a complete historical perspective on the subject.

By the structure of the spectrum of the Laplace operator on a compact Lie group, we are interested in the case  $k=2,$ and in the case were $s\geq 2.$ In this case the estimates for $r_{s,k}(R)$ are essentially deduced from the Hardy-Littlewood-Ramanujan circle method. We record the following estimates of $r_{s,2}(R)$ for $s\geq 2,$ for our further analysis.
\begin{theorem}[Representing a number as sums of squares]\label{rsk:theorem}
 Let $R\in \mathbb{N}.$ For $k=2$ and $s\in \mathbb{N},$ with $s\geq 2,$ the following statements hold.
 \begin{itemize}
    \item Let $s=2.$ Then, for any $\varepsilon>0,$ there exists $C_\varepsilon>0,$ such that
    \begin{equation}\label{Gauss}
        r_{2,2}(R)\leq C_\varepsilon R^{\varepsilon},\,R\geq 1.
    \end{equation}
    \item Let $s=3.$ Then,
    \begin{equation}\label{Bateman3}
        r_{3,2}(R)=2\pi R^{\frac{1}{2}} \mathfrak{S}_3(R),\,\,|\mathfrak{S}_3(R)|\leq C|\log(4R)(\log(\log(4R)))|,\,\,R\geq e^e/4.
    \end{equation}
    \item Let $s=4.$ Then,
    \begin{equation}\label{Bateman:4}
        r_{4,2}(R)=\pi^2 R \mathfrak{S}_4(R),\,\, R\geq e^{e},\,\,\,\limsup_{n\rightarrow \infty}\frac{\mathfrak{S}_4(R)}{\log (\log(R))}=6\pi^{-2}e^{\gamma_{*}} ,
    \end{equation}where $$\gamma_{*}=0.57721566490153286060651209008240243104215933593992\cdots,$$ is the Euler constant. 
    \item For $s\geq 5,$ there are constants $c_1(s)$ and $c_{2}(s)$ such that 
    \begin{equation}\label{Wainger}
        c_{1}(s)R^{\frac{s-2}{2}}\leq r_{s,2}(R)\leq c_{2}(s)R^{\frac{s-2}{2}}.
    \end{equation}
\end{itemize} 
\end{theorem}
 \begin{remark} For \eqref{Gauss} and a historical conjecture about the asymptotic behaviour of the terms $r_{k,k}(R)$ (the {\it  Conjecture K}) we refer the reader to Mahler \cite{Mahler1935} as well as the discussion there about the circle method.  For \eqref{Bateman3} we refer the reader to Bateman \cite[Page 84]{Bateman1951} and for   \eqref{Bateman:4} we refer the reader to Bateman \cite[Page 79]{Bateman1951}. For the historical aspects about the estimate in \eqref{Wainger} we refer the reader to Wainger \cite{Wainger2021}. In general,  the circle method establishes  that for $s$  large enough with respect to $k,$ one has the expression 
 \begin{equation}\label{ExpresionHLR}
     r_{s,k}(R)=\frac{\Gamma(1+k/s)}{\Gamma(s/k)}\mathfrak{S}_s(R)+O(R^{s/k-1-\delta}),
 \end{equation}where $\delta>0,$ and $\mathfrak{S}_s(R)$ being a singular series. In particular, for $s>2^k,$ the precise statement about \eqref{ExpresionHLR} can be found in \cite[Page 456]{IwaniecKowalski2004}. 
\end{remark}

\section{Reverse Strichartz estimates}

In this section, we analyse the Strichartz estimates on compact connected simple Lie groups. It was mentioned in the introduction, we exclude the case  $\ell=1$  since $\textnormal{Spin(3)}\cong \textnormal{SU}(2)\cong \mathbb{S}^3$ and the Strichartz estimates in the case of the spheres have been analysed in the work of Burq, Gerard, and Tzvetkov \cite{BurqGerardTzvetkov2004}.  For this, we need several preparatory useful lemmas. We start our analysis with the following result due to Sogge and H\"ormander about the $L^p$-norm for the spectral projections of the Laplacian, see \cite{Sogge1988}.
\begin{lemma}[Sogge-H\"ormander]\label{lq:l2:estimate} Let $G$ be a compact Lie group.
Let $p\in[2,\infty]$ and $P_{\lambda_\nu}:L^2(G)\rightarrow H_{\lambda_\nu}$ be the orthogonal projection onto the eigenspace $H_{\lambda_\nu}$ corresponding to $\lambda_\nu$. Then 
\begin{equation}\label{Sogge:Hormander}
    \left|\left|P_{\lambda_\nu}u_0\right|\right|_{L^p(G)}\leq C \sqrt{\lambda_\nu}^{\sigma(p)}\left|\left|P_{\lambda_\nu}u_0\right|\right|_{L^2(G)},
\end{equation}
where
\begin{equation*}
  \sigma(p):=\left\{\begin{array}{lll}
     \frac{d-1}{2}\left(\frac{1}{2}-\frac{1}{p}\right),&& \textnormal{if}\ 2\leq p\leq \frac{2(d+1)}{d-1},\\
     \\
    \frac{d-1}{2}-\frac{d}{p},&& \textnormal{if}\ \frac{2(d+1)}{d-1}<p\leq \infty.
     \end{array}\right.
\end{equation*} and with the constant $C$ in \eqref{Sogge:Hormander} depending on geometric data of $G.$ 
\end{lemma}

The following lemma will be crucial in our further analysis. We provide some crude Strichartz estimates that for a range of values of $(p,q)$ that could be non-sharp with respect to the values provided by Burq, Gerard and Tzvetkov \cite{BurqGerardTzvetkov2004}, specially if the pair $(p,q)$ is admissible. However, in the next sections we will clarify, for which values of $(p,q)$ the following lemma provides improvements of the regularity orders in \cite{BurqGerardTzvetkov2004}, in our setting.  We observe that these improvements are with respect to the established regularity orders for the Lie groups that we analyse below, but not in the setting of the manifolds considered in \cite{BurqGerardTzvetkov2004}, where their result is sharp on any manifold having closed geodesics with a common period. As it was mentioned in the introduction, we conclude that the geometry of the manifold has a strong incidence in the improvement of Strichartz estimates.

\begin{lemma}\label{Main:Lemma} Let $p\in [2,\infty),$ $q\in[2,\infty],$ and let $b$ be as in Table 1. Let $T=2\pi b \gamma.$ Then 
\begin{equation*}\label{main:estimative}
\left|\left|u\right|\right|_{L^q(G,L^p[0,T])}\lesssim\left[\sum\limits_{R=0}^\infty N_RR^{2(\alpha+\textcolor{red}{\frac{\sigma(q)}{2})}}\left|\left|P_{\frac{R}{b\gamma}}u_0\right|\right|^2_{L^2(G)}\right]^{\frac{1}{2}},
\end{equation*}
where, for any $R\in \mathbb{N},$ $N_R=\#\left\{\nu:\lambda_\nu=\frac{R}{b\gamma}\right\}$ is the number of solutions of the equation $\lambda_\nu=\frac{R}{b\gamma},$ and $\alpha=1/2-1/p$ is a loss of regularity order of the $L^p$-estimate with respect to the $L^2$-norm in the time variable. Moreover, for $p=2,$ we have the identity
\begin{equation}\label{Ort}
\forall x\in G,\, ||u(\cdot,x)||_{L^2([0,T])}=\left[\sum \limits_{R=0}^{\infty}\ 2TN_R\left|P_{\frac{R}{b\gamma}}u_0(x)\right|^2\right]^{\frac{1}{2}}.
\end{equation} 
\end{lemma}

\begin{proof}[Proof of Lemma \ref{Main:Lemma}] Since $u_0\in C^{\infty}(G)$, the solution of the linear Schr\"odinger equation is given by $$\forall t\in [0,T],\,\forall x\in G,\,\,u(t,x)=\displaystyle\sum\limits_{\nu}e^{i\lambda_{\nu}t}P_{\lambda_{\nu}}u_0(x).$$
In order to estimate the $L^p([0,T])$-norm of $u$, we shall use the Wainger-Sobolev embedding theorem, which states that for $1<r<p<\infty,$ we have that
\begin{equation*}\label{Weinger:Sobolev}
    \left|\left|\sum\limits_{n\in\mathbb{Z}}a_ne^{itn}\right|\right|_{L^p([0,2\pi])}\lesssim\left|\left|\sum_{n\in \mathbb{Z}}|n|^{\alpha}a_{n}e^{itn}\right|\right|_{L^r([0,2\pi])},\  \alpha=\frac{1}{r}-\frac{1}{p}.
\end{equation*}
In particular, when $r=2,$ and for any fixed  $x\in G,$  let us consider the function
$$   f(t)= \sum_{R=0}^\infty a_{R}e^{i tR} ,$$ defined by the Fourier coefficients
$$ a_{R}=\sum_{\nu\in \mathbb{N}^n: b\gamma\lambda_\nu=R }P_{\lambda_\nu}u_{0}(x),$$ or in other words, defined by
$$  f(t)= \sum_{R=0}^\infty a_{R}e^{i tR} =\sum_{\nu} e^{itb\gamma\lambda_\nu}P_{\lambda_\nu}u_{0}(x). $$
Then the Wainger Sobolev embedding theorem implies that
\begin{align*}\label{Brian}
    \left|\left|\sum\limits_{\nu}e^{i\lambda_{\nu}b\gamma \tau }P_{\lambda_{\nu}}u_0(x)\right|\right|_{L^p([0,2\pi])}&= \left|\left|\sum\limits_{R=0}^{\infty}a_Re^{iR\tau}\right|\right|_{L^p([0,2\pi])} \\
    & \lesssim\left|\left|\sum\limits_{R=0}^\infty R^{\alpha} a_R e^{i R \tau }\right|\right|_{L^2([0,2\pi])}\\
    &=\left|\left|\sum\limits_{\nu}e^{i\lambda_{\nu}b\gamma \tau }(\lambda_{\nu}b\gamma )^{\alpha}P_{\lambda_{\nu}}u_0(x)\right|\right|_{L^2([0,2\pi])}.
\end{align*}
 To estimate the $L^p$-norm $||u(t,x)||_{L^p([0,2\pi b\gamma])},$ let us apply  the change of variables $t\mapsto \tau,$  $t:=b\gamma\tau, $ and the estimate $L^p$-$L^2$-estimate above as follows,
\begin{align*}
  ||u(t,x)||_{L^p([0,2\pi b\gamma])} &=\left|\left|\sum\limits_{\nu}e^{i\lambda_{\nu}t}P_{\lambda_{\nu}}u_0(x)\right|\right|_{L^p([0,2\pi b\gamma])}\\
  &=(b\gamma)^{\frac{1}{p}}\left|\left|\sum\limits_{\nu}e^{i\lambda_{\nu}b\gamma \tau }P_{\lambda_{\nu}}u_0(x)\right|\right|_{L^p([0,2\pi])}  \\
  &\lesssim (b\gamma)^{\frac{1}{p}}\left|\left|\sum\limits_{\nu}e^{i\lambda_{\nu}b\gamma \tau }(\lambda_{\nu}b\gamma )^{\alpha}P_{\lambda_{\nu}}u_0(x)\right|\right|_{L^2([0,2\pi])} . 
\end{align*} The change of variables  $\tau \mapsto s,$  $s= b\gamma\tau,$ gives 
\begin{align*}
&(b\gamma)^{\frac{1}{p}}\left|\left|\sum\limits_{\nu}e^{i\lambda_{\nu}b\gamma \tau }(\lambda_{\nu}b\gamma )^{\alpha}P_{\lambda_{\nu}}u_0(x)\right|\right|_{L^2([0,2\pi])}\\
&=(b\gamma)^{\frac{1}{p}-\frac{1}{2}}\left|\left|\sum\limits_{\nu}e^{i\lambda_{\nu}s }(\lambda_{\nu}b\gamma )^{\alpha}P_{\lambda_{\nu}}u_0(x)\right|\right|_{L^2([0,2\pi b\gamma])}\\
    &=(b\gamma)^{\frac{1}{p}-\frac{1}{2}}(b\gamma)^{\alpha}\left|\left|\sum\limits_{\nu}e^{i\lambda_{\nu}s }\lambda_{\nu}^{\alpha}P_{\lambda_{\nu}}u_0(x)\right|\right|_{L^2([0,2\pi b\gamma])}\\
    &=\left|\left|\sum\limits_{\nu}e^{i\lambda_{\nu}s }\lambda_{\nu}^{\alpha}P_{\lambda_{\nu}}u_0(x)\right|\right|_{L^2([0,2\pi b\gamma])},
\end{align*}
where $\alpha=1/2-1/p.$ The last expression above can be computed as
follows:
\begin{align*}
    \left|\left|\sum\limits_{\nu}e^{i\lambda_{\nu}s}\lambda_\nu^\alpha P_{\lambda_\nu}u_0(x)\right|\right|^2_{L^2([0,T])}&=\int_0^T\left|\sum\limits_{\nu}e^{i\lambda_{\nu}s}\lambda_\nu^\alpha P_{\lambda_\nu}u_0(x)\right|^2 ds\\
    &=\int_0^T\left(\sum\limits_{\nu}e^{i\lambda_\nu s}\lambda_\nu^\alpha P_{\lambda_\nu}u_0(x)\right)\left(\sum\limits_{\nu'}e^{-i\lambda_{\nu'}s}\lambda_{\nu'}^\alpha \overline{P_{\lambda_{\nu'}}u_0(x)}\right)ds\\
    &=\int_0^T\left(\sum\limits_{\nu}\sum\limits_{\nu'}e^{i(\lambda_\nu-\lambda_{\nu'})s}(\lambda_\nu\lambda_{\nu'})^\alpha P_{\lambda_\nu}u_0(x)\overline{P_{\lambda_{\nu'}}u_0(x)}\right)ds\\
    &=\sum\limits_{\begin{subarray}{c}\nu,\nu'\\\lambda_\nu=\lambda_{\nu'}\end{subarray}}\left[\int_0^Te^{i(\lambda_\nu-\lambda_{\nu'})s}ds\right](\lambda_\nu\lambda_{\nu'})^\alpha P_{\lambda_\nu}u_0(x)\overline{P_{\lambda_{\nu'}}u_0(x)}\\
    &\ \ \ +\sum\limits_{\begin{subarray}{c}\nu,\nu'\\\lambda_\nu\neq\lambda_{\nu'}\end{subarray}}\left[\int_0^Te^{i(\lambda_\nu-\lambda_{\nu'})s}ds\right](\lambda_\nu\lambda_{\nu'})^\alpha P_{\lambda_\nu}u_0(x)\overline{P_{\lambda_{\nu'}}u_0(x)}.
\end{align*}
In the last equality, we can interchange the integration and summation symbols since $u_0\in C^{\infty}(G).$ Recalling that $T=2\pi b\gamma ,$  we have that $$\int_0^Te^{i(\lambda_\nu-\lambda_{\nu'})s}ds=\left\{\begin{array}{lll}
T,&&\textnormal{if}\ \lambda_\nu=\lambda_{\nu'},\\
0,&&\textnormal{if}\ \lambda_\nu\neq\lambda_{\nu'}.
\end{array}\right.$$ Hence, these orthogonality relations allow us to write the equalities
\begin{align*}
     \left|\left|\sum\limits_{\nu}e^{i\lambda_{\nu}s}\lambda_\nu^\alpha P_{\lambda_\nu}u_0(x)\right|\right|^2_{L^2([0,T])}&=\sum\limits_{\begin{subarray}{c}\nu,\nu'\\\lambda_\nu=\lambda_{\nu'}\end{subarray}}T(\lambda_\nu\lambda_{\nu'})^\alpha P_{\lambda_\nu}u_0(x)\overline{P_{\lambda_{\nu'}}u_0(x)}\\
     &=\sum\limits_{R=0}^{\infty}\ \sum\limits_{\nu:\lambda_\nu=
     \frac{R}{b\gamma}}\ \sum\limits_{\nu':\lambda_{\nu'}=\frac{R}{b  \gamma}} T\left(\frac{R}{b\gamma}\right)^{2\alpha}\left|P_{\frac{R}{b\gamma}}u_0(x)\right|^2\\
    &=\sum \limits_{R=0}^{\infty}\ 2T\left(\frac{R}{b\gamma}\right)^{2\alpha}\left|P_{\frac{R}{b\gamma}}u_0(x)\right|^2\sum\limits_{\nu:\lambda_\nu=\frac{R}{b\gamma}}1\\
    &=\sum \limits_{R=0}^{\infty}\ 2T\left(\frac{R}{b\gamma}\right)^{2\alpha}\left|P_{\frac{R}{b\gamma}}u_0(x)\right|^2N_R,
\end{align*}
where $$N_R:=\sum\limits_{\nu:\lambda_\nu=\frac{R}{b\gamma}}1=\#\left\{\nu:\lambda_\nu=\frac{R}{b\gamma}\right\}.$$ 
So far, we have shown that 
\begin{equation*}
 ||u(t,x)||^2_{L^p([0,T])}\lesssim \sum \limits_{R=0}^{\infty}\ 2T\left(\frac{R}{b\gamma}\right)^{2\alpha}\left|P_{\frac{R}{b\gamma}}u_0(x)\right|^2N_R,
\end{equation*}
or, equivalently,
\begin{equation*}
 ||u(t,x)||_{L^p([0,T])}\lesssim \left[\sum \limits_{R=0}^{\infty}\ 2T\left(\frac{R}{b\gamma}\right)^{2\alpha}\left|P_{\frac{R}{b\gamma}}u_0(x)\right|^2N_R\right]^{\frac{1}{2}},
\end{equation*} and also the following equality when $p=2,$
\begin{equation*}
 ||u(t,x)||_{L^2([0,T])}=\left[\sum \limits_{R=0}^{\infty}\ 2T\left|P_{\frac{R}{b\gamma}}u_0(x)\right|^2N_R\right]^{\frac{1}{2}}.
\end{equation*} 
By computing the $L^q(G)$-norm of $||u(t,x)||_{L^p([0,T])}$ we have that
\begin{align*}
    ||u(t,x)||_{L^q(G,L^p[0,T])}&=\left\{\int_{G}||u(t,x)||^q_{L^p([0,T])}dx\right\}^{\frac{1}{q}}\\
    &\lesssim \left\{\int_{G}\left[\sum \limits_{R=0}^{\infty}\ 2T\left(\frac{R}{b\gamma}\right)^{2\alpha}\left|P_{\frac{R}{b\gamma}}u_0(x)\right|^2N_R\right]^{\frac{q}{2}}dx\right\}^{\frac{1}{q}}\\
    &\leq \left\{\sum\limits_{R=0}^\infty\left[\int_{G}\left(2T\left(\frac{R}{b\gamma}\right)^{2\alpha}N_R\left|P_{\frac{R}{b\gamma}}u_0(x)\right|^2\right)^{\frac{q}{2}}dx\right]^{\frac{2}{q}}\right\}^{\frac{1}{2}},
\end{align*}
where the last inequality is nothing but the Minkowski integral inequality. Note that 
    \begin{align*}
   &\left\{\sum\limits_{R=0}^\infty\left[\int_{G}\left(2T\left(\frac{R}{b\gamma}\right)^{2\alpha}N_R\left|P_{\frac{R}{b\gamma}}u_0(x)\right|^2\right)^{\frac{q}{2}}dx\right]^{\frac{2}{q}}\right\}^{\frac{1}{2}}\\
   &=\left\{\sum_{R=0}^{\infty}2T\left(\frac{R}{b\gamma}\right)^{2\alpha}N_R\left[\int_{G}\left|P_{\frac{R}{b\gamma}}u_0(x)\right|^qdx\right]^{\frac{2}{q}}\right\}^{\frac{1}{2}}\\
    &=\left\{\sum_{R=0}^{\infty}2T\left(\frac{R}{b\gamma}\right)^{2\alpha}N_R\left|\left|P_{\frac{R}{b\gamma}}u_0\right|\right|^2_{L^q(G)}\right\}^{\frac{1}{2}},
\end{align*}
and, by Lemma \ref{lq:l2:estimate}, we know that $\left|\left|P_{\frac{R}{b\gamma}}u_0\right|\right|^2_{L^q(G)}\lesssim \left(\frac{R}{b\gamma}\right)^{\textcolor{red}{\sigma(q)}}\left|\left|P_{\frac{R}{b\gamma}}u_0\right|\right|^2_{L^2(G)}$ so that
\begin{align*}
    &\left\{\sum_{R=0}^{\infty}2T\left(\frac{R}{b\gamma}\right)^{2\alpha}N_R\left|\left|P_{\frac{R}{b\gamma}}u_0\right|\right|^2_{L^q(G)}\right\}^{\frac{1}{2}}\\
    &\lesssim\left\{\sum_{R=0}^{\infty}2T\left(\frac{R}{b\gamma}\right)^{2\alpha}N_R\left(\frac{R}{b\gamma}\right)^{\textcolor{red}{\sigma(q)}}\left|\left|P_{\frac{R}{b\gamma}}u_0\right|\right|^2_{L^2(G)}\right\}^{\frac{1}{2}}\\
    &=\frac{\sqrt{2T}}{(b\gamma)^{(\alpha+\textcolor{red}{\frac{\sigma(q)}{2}})}}\left[\sum\limits_{R=0}^\infty N_RR^{2(\alpha+\textcolor{red}{\frac{\sigma(q)}{2}})}\left|\left|P_{\frac{R}{b\gamma}}u_0\right|\right|^2_{L^2(G)}\right]^{\frac{1}{2}}.
\end{align*}
Hence
\begin{equation*}
    ||u(t,x)||_{L^q(G,L^p[0,T])}\lesssim \frac{\sqrt{2T}}{(b\gamma)^{(\alpha+\textcolor{red}{\frac{\sigma(q)}{2}})}}\left[\sum\limits_{R=0}^\infty N_RR^{2(\alpha+\textcolor{red}{\frac{\sigma(q)}{2}})}\left|\left|P_{\frac{R}{b\gamma}}u_0\right|\right|^2_{L^2(G)}\right]^{\frac{1}{2}},
\end{equation*}
and the proof is complete.
\end{proof}

\begin{theorem}[{Reverse Strichartz estimates for compact simple Lie groups}]\label{LemmaSpin2l:1} Let {$p\in[2,\infty)$ and $q\in[2,\infty].$} Let $G$ be a compact connected simple Lie group of rank $\ell\geq 2$, let $b$ as in Theorem \ref{general:spectrum:theorem}, and let $m$ be defined by $$m=\left\{\begin{array}{lll}
\ell+1, &&\textnormal{if}\ \tilde{G}\cong\textnormal{SU($\ell+1$)}\ \textnormal{or}\ \tilde{G}\cong G_2,\\
\ell, &&\textnormal{otherwise},
\end{array}\right.$$ where $\tilde{G}$ is the universal cover of $G$.
Let $T=2\pi b  \gamma.$ Then 
\begin{equation}\label{main:estimative:Spin2l:1}
\left|\left|u\right|\right|_{L^q(G,L^p[0,2\pi b  \gamma])}\lesssim_s  \Vert u_0\Vert_{H^{ s} (G) },
\end{equation} 
where $s\geq  \frac{m}{2}-\frac{2}{p}+\sigma(q)$ if $m\geq 5,$ $s>\frac{m}{2}-\frac{2}{p}+\sigma(q)$ if $m=2,3,4$ and 
\begin{equation*}
    \sigma(q):=\left\{\begin{array}{lll}
     \frac{d-1}{2}\left(\frac{1}{2}-\frac{1}{q}\right),&& \textnormal{if}\ 2\leq q\leq \frac{2(d+1)}{d-1},\\
     \\
     \frac{d-1}{2}-\frac{d}{q},&& \textnormal{if}\ \frac{2(d+1)}{d-1}<q\leq \textcolor{red}{\infty.}
     \end{array}\right.
\end{equation*} 
\end{theorem}

\begin{proof} Let $G$ be a compact connected simple Lie group of rank $\ell\geq 2$. 
By Theorem \ref{general:spectrum:theorem}, there exist integers $b,R_0$ and an integer matrix of rank $\ell$ such that every eigenvalue of the Laplacian $\mathcal{L}_{G}$ can be written as
\begin{equation}\label{spectrum:spinl:2}
   \lambda_{\nu}= \frac{1}{b\gamma}\left(n_1^2+\cdots+n_m^2-R_0\right), 
\end{equation}
where $[n_1\ \cdots\ n_m]^T=A[\nu_1\ \cdots \nu_\ell]^T$ for some $(\nu_1,...,\nu_\ell)\in\mathbb{N}^\ell.$ On the other hand,  in view of Lemma \ref{Main:Lemma} we have the following estimate
\begin{equation}\label{To:estimate:Bl}
\left|\left|u\right|\right|_{L^q(G,L^p[0,2\pi b  \gamma])}\lesssim\left[\sum\limits_{R=0}^\infty N_RR^{2(\alpha+\textcolor{red}{\frac{\sigma(q)}{2}})}\left|\left|P_{\frac{R}{b\gamma}}u_0\right|\right|^2_{L^2(G)}\right]^{\frac{1}{2}},
\end{equation}
with 
$$N_R=\#\left\{\nu:\lambda_\nu=\frac{R}{b\gamma}\right\}$$ and $\alpha=1/2-1/p$. To estimate $N_R$ observe that since $A$ has maximal rank equal to $\ell$, then the linear map $T_A:\nu\mapsto A\nu,$ from $\mathbb{R}^\ell$ to $\mathbb{R}^m$ is injective. 
Define the exceptional set 
\begin{equation}
   \mathscr{N}_R= \{Y=(n_i)_{1\leq i\leq m}\in \textnormal{Im}(T_A): n_1^2+\cdots+n_m^2=R+R_0,\,\,\textnormal{ and } A\nu=Y \Longrightarrow \nu \in\mathbb{N}^\ell\}.
\end{equation}Since there is a bijective correspondence between elements $Y=(n_1,n_2,\cdots ,n_m)\in \mathscr{N}_R$ with solutions of the equation $\lambda_{\nu}=\frac{R}{b\gamma},$ that indeed is provided by the identity $A\nu=Y$, we have that
\begin{equation}
    N_R=\#\left\{\nu:\lambda_\nu=\frac{R}{b\gamma}\right\}=\# \mathscr{N}_R\leq r_{m,2}(R+R_0).
\end{equation} 
Recalling Theorem \ref{rsk:theorem} we obtain the following inequalities for any $\varepsilon>0$:
\begin{align*}
        \bullet\ r_{2,2}(R+R_0)\lesssim_\varepsilon&(R+R_0)^\varepsilon\lesssim_\varepsilon R^\epsilon,\\
        \bullet\ r_{3,2}(R+R_0)\lesssim \ &(R+R_0)^{\frac{1}{2}}[\log(4(R+R_0))(\log(\log(4(R+R_0)))]\\
        \lesssim_\varepsilon&(R+R_0)^{\frac{1}{2}+\varepsilon}\\
        \lesssim_\varepsilon&R^{\frac{1}{2}+\varepsilon},\\
        \bullet\ r_{4,2}(R+R_0)\lesssim\ &(R+R_0)\log(\log(R+R_0))\lesssim_\varepsilon (R+R_0)^{1+\varepsilon}\lesssim_\varepsilon R^{1+\varepsilon},\\
        \bullet\ r_{m,2}(R+R_0)\lesssim\ &(R+R_0)^{\frac{m-2}{2}}\lesssim R^{\frac{m-2}{2}}
        \end{align*}
or, equivalently, $r_{m,2}(R+R_0)\lesssim R^{\frac{m-2}{2}+\varepsilon}$ for all $\varepsilon>0$ if $m=2,3,4$ and for all $\varepsilon\geq 0$  if $m\geq 5.$

Now, we can estimate the right hand side of \eqref{To:estimate:Bl} as follows
\begin{align*}
   & \left[\sum\limits_{R=0}^\infty N_RR^{2(\alpha+\textcolor{red}{\frac{\sigma(q)}{2}})}\left|\left|P_{\frac{R}{b\gamma}}u_0\right|\right|^2_{L^2(G) }\right]^{\frac{1}{2}}\\
    &\lesssim \left[\sum\limits_{R=0}^\infty R^{2\left(\frac{\varepsilon}{2}+\frac{m-2}{4}+\alpha+\textcolor{red}{\frac{\sigma(q)}{2}}\right)}\left|\left|P_{\frac{R}{b\gamma}}u_0\right|\right|^2_{L^2(G)}\right]^{\frac{1}{2}}\\
    &\asymp \Vert u_0\Vert_{H^{\varepsilon+\frac{m-2}{2}+2\alpha+\sigma(q)  }}.
\end{align*}    Observing that $\alpha=1/2-1/p,$ and that 
\begin{equation}
    \frac{m-2}{2}+2\alpha+\sigma(q)=\frac{m}{2}-\frac{2}{p}+\sigma(q)= \frac{m}{2}-\frac{2}{p}+\sigma(q),
\end{equation}
we have proved \eqref{main:estimative:Spin2l:1}. The proof is complete.
\end{proof}
\begin{remark} Note that $m$ always can be written in terms of the dimension $d$ of $G.$ For instance, when $G=\textnormal{Spin}(2\ell+1)$ we have that $d=\ell(2\ell+1),$ thus
$$m=\ell=\frac{-1+\sqrt{1+8d}}{4},$$ and
the inequality for the regularity order $s$ in \eqref{main:estimative:Spin2l:1} can be written as  \begin{equation}
    s\geq \frac{\sqrt{8d+1}-1}{8}-\frac{2}{p}+\sigma(q).
\end{equation} In terms of the dimension $d=\ell(\ell+2)$ of $\textnormal{SU}(\ell+1),$ we can write the inequality for the regularity order in  \eqref{main:estimative:Spin2l:1}  as follows
\begin{equation}
    {s\geq  \frac{-1+\sqrt{1+d}+1}{2}-\frac{2}{p}+\sigma(q)=\frac{\sqrt{1+d}}{2}-\frac{2}{p}+\sigma(q).}
\end{equation}  As for the case of $\textnormal{Spin}(2\ell)$ whose  dimension is $d=\ell(2\ell-1),$ we have that
$$\ell=\frac{1+\sqrt{1+8d}}{4},$$ and
the inequality for the regularity order $s$ in \eqref{main:estimative:Spin2l:1} can be written as  \begin{equation}
    {s\geq \frac{1+\sqrt{8d+1}}{8}-\frac{2}{p}+\sigma(q).}
\end{equation}   Note that in the case of $\textnormal{Spin}(2\ell)$ which has  dimension $d=\ell(2\ell-1),$ we have that
$$\ell=\frac{1+\sqrt{1+8d}}{4},$$ and
the inequality for the regularity order $s$ in \eqref{main:estimative:Spin2l:1} can be written as  \begin{equation}
    {s\geq \frac{1+\sqrt{8d+1}}{8}-\frac{2}{p}+\sigma(q).}
\end{equation}  Also, in terms of the dimension $d=\ell(2\ell+1)$ of $\textnormal{Sp}(\ell),$ we have that
$$ \ell= \frac{-1+\sqrt{1+8d}}{4}, $$ and then, we
can write the inequality for the regularity order in  \eqref{main:estimative:Spin2l:1}  as follows
\begin{equation}
    {s\geq \frac{-1+\sqrt{1+8d}}{8}-\frac{2}{p}+\sigma(q).}
 \end{equation}   
\end{remark}

Next, we observe that from the reverse Strichartz estimates in Theorem \ref{LemmaSpin2l:1} one can deduce classical Strichartz estimates on $G.$

\begin{corollary}[Strichartz estimates on compact Lie groups] Let { $2\leq q\leq p<\infty.$} Let $G$ be a compact connected simple Lie group of rank $\ell\geq 2$, let $b$ as in Theorem \ref{general:spectrum:theorem}, and let $m$ be defined by $$m=\left\{\begin{array}{lll}
\ell+1, &&\textnormal{if}\ \tilde{G}\cong\textnormal{SU($\ell+1$)}\ \textnormal{or}\ \tilde{G}\cong G_2,\\
\ell, &&\textnormal{otherwise},
\end{array}\right.$$ where $\tilde{G}$ is the universal cover of $G$.
Let $T=2\pi b  \gamma.$ Then 
\begin{equation}\label{corolary:th}
\left|\left|u\right|\right|_{L^p([0,2\pi\gamma b],L^q(G))}\lesssim_s  \Vert u_0\Vert_{H^{ s} (G) },
\end{equation} 
where $s\geq  \frac{m}{2}-\frac{2}{p}+\sigma(q)$ if $m\geq 5,$ $s>\frac{m}{2}-\frac{2}{p}+\sigma(q)$ if $m=2,3,4$ and 
\begin{equation*}
    \sigma(q):=\left\{\begin{array}{lll}
     \frac{d-1}{2}\left(\frac{1}{2}-\frac{1}{q}\right),&& \textnormal{if}\ 2\leq q\leq \frac{2(d+1)}{d-1},\\
     \\
     \frac{d-1}{2}-\frac{d}{q},&& \textnormal{if}\ \frac{2(d+1)}{d-1}<q\leq \textcolor{red}{\infty.}
     \end{array}\right.
\end{equation*} 
 \begin{proof} Let us use the Minkowski integral inequality with $p/q\geq 1.$ Indeed, observe that
 \begin{align*}
    \left|\left|u\right|\right|_{L^p([0,2\pi\gamma b],L^q(G))} &=\left(\smallint_{0}^{2\pi \gamma \beta}\Vert u(t,\cdot)\Vert^p_{L^q(G)} dt\right)^{1/p}=\left(\left(\smallint_{0}^{2\pi \gamma \beta}\Vert u(t,\cdot)\Vert^p_{L^q(G)}dt\right)^{q/p}\right)^{1/q} \\
    &= \left(\left(\smallint_{0}^{2\pi \gamma \beta}\left(\smallint_{G} |u(t,x)|^qdx\right)^{p/q}dt\right)^{q/p}\right)^{1/q}\\
    &\leq \left( \smallint_{G}\left(\smallint_{0}^{2\pi \gamma \beta} |u(t,x)|^{q\times p/q }dt\right)^{q/p} dx\right)^{1/q}\\
    &= \left( \smallint_{G}\Vert u(\cdot,x)\Vert^q_{L^p([0,2\pi \gamma b])}dx \right)^{1/q}=\left|\left|u\right|\right|_{L^q(G,L^p[0,2\pi b  \gamma])}.
 \end{align*}In view of the previous inequality, it follows from \eqref{main:estimative:Spin2l:1} the estimate in \eqref{corolary:th}. The proof is complete.     
 \end{proof}   
\end{corollary}

\subsection{Data statement} No data were obtained during the realization of this work.
\bibliographystyle{amsplain}

\end{document}